\documentclass{amsart}

\RequirePackage{color}

\def\smallskip{\vskip\smallskipamount}
\def\medskip{\vskip\medskipamount}
\def\bigskip{\vskip\bigskipamount}

\usepackage{amsmath,amsthm,bm,mathrsfs,amscd,tikz, hyperref}
\usepackage{ytableau}
\usepackage{comment}
\usepackage{mathdots}
\usepackage[all]{xy}
\usepackage{graphicx}
\usepackage[colorinlistoftodos]{todonotes}
\usepackage{enumerate}
\usepackage{feynmp-auto}
\usepackage[english]{babel}
\usepackage[utf8]{inputenc}
\usepackage{float}
\usepackage{tikz}
\usepackage{tikz-cd}
\usepackage{amssymb}
\usepackage[maxbibnames=99]{biblatex}
\usepackage{mathtools, csquotes}
\usetikzlibrary{decorations.pathmorphing}
\usepackage{caption}
\usepackage{stmaryrd}
\usepackage{xcolor}

\usepackage{stackengine}
\setstackEOL{\cr}

\usepackage{graphicx}

\usepackage[utf8]{inputenc}

\newtheoremstyle{thmstyle}{}{}{\itshape}{}{\bfseries}{ }{5pt}{}
\newtheoremstyle{exstyle}{}{}{}{}{\bfseries}{ }{5pt}{}
\newtheoremstyle{defstyle}{}{}{}{}{\bfseries}{ }{5pt}{}
\newtheoremstyle{remstyle}{}{}{}{}{\bfseries}{ }{5pt}{}

\theoremstyle{thmstyle}
\newtheorem{thm}{Theorem}[section]
\newtheorem{theorem}[thm]{Theorem}

\newtheorem*{theorem*}{Theorem}
\newtheorem{lemma}[thm]{Lemma}

\newtheorem{proposition}[thm]{Proposition}

\newtheorem{corollary}[thm]{Corollary}

\newtheorem*{definition*}{Definition}
\newtheorem*{corollary*}{Corollary}

\theoremstyle{exstyle}

\newtheorem{example}[thm]{Example}

\theoremstyle{defstyle}

\newtheorem{definition}[thm]{Definition}

\newtheorem{def-prop}[thm]{Definition-Proposition}
\newtheorem{def-lem}[thm]{Definition-Lemma}
\newtheorem{rem-convention}[thm]{Remark-Convention}
\newtheorem{def-note}[thm]{Definition-Notation}

\theoremstyle{remstyle}

\newtheorem{remark}[thm]{Remark}

\theoremstyle{remstyle}

\newcommand{\ZZ}{\mathbb Z}
\newcommand{\M}{\mathcal{M}}
\newcommand{\C}{\mathcal{C}}
\newcommand{\X}{\mathcal{X}}
\newcommand{\Y}{\mathcal{Y}}
\newcommand{\T}{\mathcal{T}}

\newcommand{\F}{\mathcal{F}}
\newcommand{\D}{\mathcal{D}}
\newcommand{\U}{\mathcal{U}}
\newcommand{\V}{\mathcal{V}}

\newcommand{\Ho}{\mathrm{H}}
\newcommand{\K}{\operatorname{K}}

\newcommand{\g}{\mathbf g}
\newcommand{\s}{\mathfrak s}

\newcommand{\Eb}{\mathbb{E}}

\newcommand{\La}{\Lambda}

\newcommand{\add}{\operatorname{add}}
\newcommand{\tors}{\operatorname{tors}}
\newcommand{\ftors}{\operatorname{f-tors}}
\newcommand{\lwide}{\operatorname{l-wide}}
\newcommand{\fptors}{\operatorname{(f,p)-tors}}
\newcommand{\ptors}{\operatorname{p-tors}}
\newcommand{\ithick}{\operatorname{inj-thick}}
\newcommand{\xtors}{\operatorname{x-tors}}
\newcommand{\ytors}{\operatorname{y-tors}}
\newcommand{\covtors}{\operatorname{cov-tors}}
\newcommand{\contrtors}{\operatorname{contr-tors}}
\newcommand{\thick}{\operatorname{thick}}

\newcommand{\xytors}{\operatorname{(x,y)-tors}}

\newcommand{\cotors}{\operatorname{cotors}}
\newcommand{\ccotors}{\operatorname{c-cotors}}

\newcommand{\hcotors}{\operatorname{h-cotors}}
\newcommand{\contrhcotors}{\operatorname{(contr,h)-cotors}}
\newcommand{\sttilt}{\operatorname{s\tau-tilt}}
\newcommand{\chcotors}{\operatorname{(c,h)-cotors}}

\newcommand{\stors}{\operatorname{stors}}

\newcommand{\silt}{\operatorname{silt}}

\newcommand{\proj}{\operatorname{proj}}
\newcommand{\inj}{\operatorname{inj}}
\newcommand{\mmod}{\operatorname{mod}}

\newcommand{\Ext}{\mathrm{Ext}}
\newcommand{\Hom}{\mathrm{Hom}}

\newcommand{\Ker}[1]{\mathrm{Ker}(#1)}

\graphicspath{ {./images/} }

\makeatletter
\newcommand{\doublewidetilde}[1]{{%
  \mathpalette\double@widetilde{#1}%
}}
\newcommand{\double@widetilde}[2]{%
  \sbox\z@{$\m@th#1\widetilde{#2}$}%
  \ht\z@=.9\ht\z@
  \widetilde{\box\z@}%
}
\makeatother
\newcommand{\Address}{{
  \bigskip
  \footnotesize

  E.~Gupta, \textsc{Université Paris-Saclay, UVSQ, CNRS, Laboratoire de Mathématiques de Versailles, 78000, Versailles, France.}\par\nopagebreak
  \textit{E-mail address}: \texttt{esha.gupta2@uvsq.fr}
}}

\title[Silting objects and (co) torsion pairs in truncated derived categories]{On $d$-term silting objects and (co) torsion pairs in truncated derived categories}
\author{Esha Gupta}
\date{}

\begin{document}
\begin{abstract}
    For a finite-dimensional algebra $\La$ over an algebraically closed field $K$, it is known that the poset of $2$-term silting objects in $\K^b(\proj\La)$ is isomorphic to the poset of functorially finite torsion classes in $\mmod\La$, and to that of complete cotorsion classes in $\K^{[-1,0]}(\proj\La)$. In this work, we generalise this result to the case of $d$-term silting objects for arbitrary $d\geq 2$ by introducing the notion of torsion classes for extriangulated categories. In particular, we show that the poset of $d$-term silting objects in $\K^b(\proj\La)$ is isomorphic to the poset of complete and hereditary cotorsion classes in $\K^{[-d+1,0]}(\proj\La)$, and to that of positive and functorially finite torsion classes in $\D^{[-d+2,0]}(\mmod\La)$, an extension-closed subcategory of $\D^b(\mmod\La)$. We further show that the posets $\cotors\K^{[-d+1,0]}(\proj\La)$ and $\tors \D^{[-d+2,0]}(\mmod\La)$ are lattices, and that the truncation functor $\tau_{\geq -d+2}$ gives an isomorphism between the two.
\end{abstract}
\maketitle

\tableofcontents{}

\section{Introduction}
The theory of $\tau$-tilting was first introduced in \cite{AIR} as a mutation completion of tilting theory. In that seminal work, the authors provided isomorphisms between the following posets for an arbitrary finite-dimensional algebra $\La$.
\begin{enumerate}
    \item Isomorphism classes of basic support $\tau$-tilting modules in $\mmod\La$.
    \item Functorially finite torsion pairs in $\mmod\La$.
    \item Isomorphism classes of basic $2$-term silting objects in $\K^b(\proj\La)$.
\end{enumerate}

For path algebras of acyclic quivers, it was proven in \cite{IT} that finitely generated torsion classes in $\mmod\La$ are in bijection with finitely generated wide subcategories of $\mmod\La$. This was generalized in \cite{MS} to all finite-dimensional algebras, adding the following class to the above list. 
\begin{enumerate}
    \item [(4)] Left finite wide subcategories of $\mmod \La$.
\end{enumerate}

In \cite{G1}, the author interpreted the class of thick subcategories as the `mirror' of wide subcategories, adding the following class to the mix.
\begin{enumerate}
    \item [(5)] Thick subcategories in $\K^{[-1,0]}(\proj \La)$ with enough injectives.
\end{enumerate}

Finally, in \cite{PZ}, the authors gave a map from the poset of torsion classes in $\mmod\La$ and the poset of cotorsion pairs in $\K^{[-1,0]}(\proj\La)$, which restricts to an isomorphism of the above posets with
\begin{enumerate}
    \item [(6)] Complete cotorsion pairs in $\K^{[-1,0]}(\proj\La)$. 
\end{enumerate}
Recently, in \cite{G2}, the same map was shown to provide a bijection between all torsion pairs in $\mmod\La$ and all cotorsion pairs in $\K^{[-1,0]}(\proj\La)$. The proof of this fact relies on the following equivalence of categories. 
\begin{equation}\label{H0}
    \Ho^0:\frac{\K^{[-1,0]}(\proj\La)}{\Sigma \proj\La}\to \mmod\La
\end{equation} 

The above isomorphisms can be summarized in the diagram below, where $\K_\La=\K^{[-1,0]}(\proj\La)$.
\begin{figure}[h!]
\captionsetup{labelformat=empty}
\centering
\begin{tikzcd}
{}&{}& {2\mbox{-}\silt \K_\La} \arrow[rd,"\textrm{\cite{G1}}",sloped]\arrow[d,no head] & {}& {}  \\
{\K_\La}& {c\mbox{-}\cotors \K_\La} \arrow[ru] \arrow[rr] \arrow[ddd,"\textrm{\cite{PZ}}"] & {}\arrow[d, no head]& {\ithick \K_\La} \arrow[ddd, "\textrm{\cite{G1}}"] & {}\\
{} \arrow[r, no head, dashed] & {} \arrow[rr, no head, dashed]  & {} \arrow[d, "\textrm{\cite{AIR}}"] & {} \arrow[r, no head, dashed] & {} \\
{}& {}& {\sttilt\La} \arrow[rd, "\textrm{\cite{MS}}", sloped] \arrow[ld, "\textrm{\cite{AIR}}", sloped] & {}& {\mmod\La} \\
{}& \ftors\La \arrow[rr,"\textrm{\cite{MS}}"]& {} & \lwide\La &  {}
\end{tikzcd}  
\caption{Figure taken from \cite{G1}}
\end{figure}
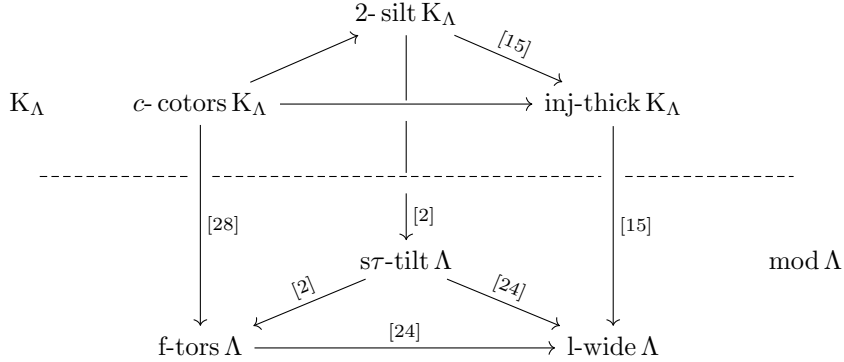

Given the variety of families in bijection with $2$-term silting objects in $\K^b(\proj\La)$, it is natural to consider if the poset of $d$-term silting objects in $\K^b(\proj\La)$ enjoys similar properties for arbitrary $d\geq 2$. It was proven in \cite[Corollary~3.4]{IJ} that this poset is isomorphic to the poset of simple-minded collections in $\D^b(\mmod\La)$ lying in $\D^{[-d+1,0]}(\mmod\La)$. Furthermore, they show that if $\La$ is hereditary of Dynkin type, silting objects with cohomology concentrated in $[-d+1,0]$ are in bijection with $d$-simple-minded systems in the $(-d)$-cluster category $\mathcal{C}_{-d}(\La)$, which are then proven to be counted by the positive Fuss-Catalan number $C^+_d(W)$, where $W$ is the Weyl group of $\La$ \cite[Theorem~1.1, 1.2]{IJ}. 
\medskip
For an acyclic quiver $Q$, it was shown in \cite[Theorem~5.14]{KQ} that the hearts of $\D^b(\mmod kQ)$ lying between the canonical heart $\mathcal{H}_Q=\mmod kQ$ and $\mathcal{H}_Q[d-1]$ are in bijection with basic $d$-term silting complexes as well as $d$-cluster tilting sets.

\medskip
In \cite{AHJKPT2}\footnote{This is a work in progress building up on the previous work \cite{AHJKPT1} of the same authors, and whose results were announced at the online talk https://www.fd-seminar.xyz/talks/2023-05-04/}, the authors study the subset of $d$-term silting objects whose cohomology is concentrated in degrees $0$ and $-d+1$ using $(d-1)$-torsion classes in a $(d-1)$-cluster tilting subcategory $\mathcal{M}$ of $\mmod\La$. In particular, they show that the set of functorially finite $(d-1)$-torsion classes in $\mathcal{M}$ injects into this set.

\medskip
Furthermore, for Dynkin types, $d$-term silting objects were shown to be in bijection with several combinatorial objects in \cite{STW}, including $d$-eralized clusters and $d$-eralized noncrossing partitions. These objects were then shown to be counted by the Fuss-Catalan numbers of type $W$, where $W$ is the Coxeter group associated to $\La$.

In this work, we start by generalising the notion of torsion pairs for abelian categories \cite{D} and triangulated categories \cite{BR, IJ} to extriangulated categories with an additional structure that plays the role of extension spaces in all integer degrees satisfying some compatibility conditions. This is axiomatised in the definition of a \emph{bivariant $\delta$-functor} in \cite[Definition~4.5]{GNP}. Although the definition of a $\delta$-functor is more general, for simplification, we restrict ourselves to the case where the degree $0$ extensions are given by the $\Hom$-functors while the degree $1$ extensions coincide with the extriangulated structure. We also introduce the notion of positive torsion pairs for such categories. For extriangulated categories equipped with negative first extensions, there exists the notion of $s$-torsion classes as introduced in \cite{AET}, and we provide a comparison between these two notions in Theorem \ref{bijections}. 

Recall that for an extriangulated category $(\C,\Eb,\s)$, for any $\Eb$-extension $\delta\in \Eb(C,A)$, there are canonical natural transformations $A_\delta:\Hom(-,C)\to \Eb(-,A)$ and $B_\delta:\Hom(A,-)\to \Eb(C,-)$ \cite[Definition~3.1]{NP}. 

\begin{definition*}(Definition \ref{torclass})
    Let $(\C,\Eb,\s)$ be an extriangulated category equipped with a bivariant $\delta$-functor $(G^\bullet, \delta_\#, \delta^\#)$ such that $G^0=\Hom_\C$, $G^1=\Eb$, and for any $\Eb$-extension $\delta\in \Eb(C,A)$, $\delta^0_\#=A_\delta$ and $\delta^{\#,0}=B_\delta$. Let $\T,\F$ be full subcategories of $\C$. 
    \begin{enumerate}
        \item The pair $(\T,\F)$ is called a \emph{torsion pair} if 
        \begin{enumerate}
        \item $\Hom(t,\F)=0$ if and only if $t\in\T$;
        \item $\Hom(\T,f)=0$ if and only if $f\in\F$.
        \end{enumerate}
        \item A torsion pair $(\T,\F)$ is called \emph{positive} if $G^{k}(\T,\F)=0$ for all $k\leq -1$.
        \end{enumerate}
\end{definition*}

In order to relate this notion to $d$-term silting objects, the main step is finding a `higher' generalization of Equation \ref{H0}. It is natural to replace $\K^{[-1,0]}(\proj\La)$ by $\K^{[-d+1,0]}(\proj\La)$, and the key step is replacing $\mmod\La$ by $\D^{[-d+2,0]}(\mmod\La)$. These two categories are then related by the truncation functor $\tau_{\geq -d+2}$ and we have the following equivalence (Proposition \ref{equiv}).
\begin{equation}\label{tau_d}
    \tau_{\geq -d+2}:\frac{\K^{[-d+1,0]}(\proj\La)}{\Sigma^{d-1} \proj\La}\to \D^{[-d+2,0]}(\mmod\La)
\end{equation}

Our main results are the following. 

\begin{theorem*}(\ref{bijections})
    The map $\Phi$ defined as $(\X,\Y)\mapsto (\tau_{\geq -d+2}\Y,~(\tau_{\geq -d+2}\Y)^{\perp})$ gives an isomorphism between the poset of cotorsion pairs in $\K^{[-d+1,0]}(\proj\La)$ and the poset of torsion pairs in $\D^{[-d+2,0]}(\mmod\La)$.
\end{theorem*}
In \cite{KY}, the authors provided isomorphisms $\psi$ (resp. $\psi'$) between the poset of silting objects in $\K^b(\proj\La)$ and algebraic $t$-structures in $\D^b(\mmod\La)$ (resp. bounded co-$t$-structures in $\K^b(\proj\La)$) (See Definitions \ref{psi} and \ref{psi2}).

\begin{theorem*}(\ref{main})
    The map $\Phi$ defined above along with the maps $\psi$ and $\psi'$ introduced in \cite{KY} restrict to the following commutative triangle of poset isomorphisms.  
    \begin{equation*}
    \begin{tikzcd}
    d\mbox{-}\silt\La \arrow[rr, "\psi"] \arrow[rrd, "\psi'"'] &  & \chcotors \K^{[-d+1,0]}(\proj\La)\arrow[d, "\Phi"] \\ &  & \fptors \D^{[-d+2,0]}(\mmod\La)                     
    \end{tikzcd}
    \end{equation*}
\end{theorem*}

The above results can be summarized in the following commutative diagram of poset morphisms, where c stands for complete, h for hereditary, f for functorially finite, p for positive, cov for covariantly, and contr for contravariantly finite. 
\begin{center}
\begin{tikzcd}[column sep=2em, /tikz/column 5/.style={column sep= 0.1em}, /tikz/column 4/.style={column sep=0.5em}]
& {\chcotors\La} \arrow[rd, hook] \arrow[ddd, "\sim"] \arrow[rrr, hook] &  & & {\cotors\La} \arrow[ddd, "\sim"] & 
\\ & & {\contrhcotors\La} \arrow[ddd, "\sim"] \arrow[rru, hook] & \mbox{hereditary} \arrow[<->]{d} &  & \mbox{complete} \arrow[<->]{d} \\
d\mbox{-}\silt\La \arrow[ruu, "\sim"] \arrow[rd, "\sim"'] & &  & \mbox{positive}  &  & \mbox{cov. finite} \\
& {\fptors\La} \arrow[rd, hook] \arrow[rrr, hook]     & & & {\tors\La} &  \\ 
& & {\stors\La} \arrow[rru, hook] & & &
\end{tikzcd}   
\end{center}

Furthermore, it is known that the poset of torsion classes in $\mmod\La$ is a lattice, and its lattice-theoretic properties have been extensively studied (\cite{DIRRT}, \cite{T}, \cite{DIJ}). The corresponding result also holds for torsion and cotorsion pairs in extriangulated categories.
\begin{theorem*}(\ref{torslattice})
    The poset of torsion pairs in an extriangulated category $\C$ is a lattice.
\end{theorem*}
\begin{theorem*}(\ref{cotorslattice})
    The poset of cotorsion pairs in an extriangulated category $\C$ is a lattice.
\end{theorem*}

We also have the following result which implies that the poset of $d$-term silting objects embeds into a smaller lattice.

\begin{theorem*}(\ref{ptorslattice})
    The subposet of positive torsion pairs in $ \D^{[-d+2,0]}(\mmod\La)$ is a sublattice of $\tors\D^{[-d+2,0]}(\mmod\La)$.
\end{theorem*}
\begin{corollary*}(\ref{hcotorslattice})
    The subposet of hereditary cotorsion pairs in \linebreak $ \K^{[-d+1,0]}(\proj\La)$ is a sublattice of $\cotors\K^{[-d+1,0]}(\proj\La)$.
\end{corollary*}

Since the class of $s$-torsion classes is contained in the class of positive torsion classes, it would be interesting to know if it is closed under meets and joins, or if it forms a lattice. 

\medskip
In \S~\ref{cluster-tilting}, we will provide a relation between $(d-1)$-torsion classes studied in \cite{AHJKPT2} and positive torsion classes in $\D^{[-d+2,0]}(\mmod\La)$.

\section*{Acknowledgements}
I would like to express my most sincere gratitude towards Pierre-Guy Plamondon for his guidance, availability, and constant support throughout the development of this work. I am also thankful to Vincent Pilaud for pointing out the non-semidistributivity of the lattice of silting objects (Example \ref{nsemid}), to Hugh Thomas for referring to the study of $d$-term silting objects for Dynkin cases \cite{STW}, and to Yann Palu for his several interesting questions, including the one about the relation of this work to higher torsion classes. I would also like to thank Monica Garcia for sharing the manuscript of her Ph.D. thesis, which provided the idea for the proof of Theorem \ref{bijections}. Finally, I would like to thank Søndre Kvamme for his helpful discussion and all the authors of \cite{AHJKPT2} for sharing the preliminary version of their article whose results have been used in \S~\ref{cluster-tilting}. 

\section{Notation and definitions}\label{not}
Throughout this work, $\La$ will denote a finite-dimensional algebra over an algebraically closed field $K$. For a morphism $f$ in a triangulated category, $C(f)$ will denote the cone of $f$. We denote by $\D^b(\mmod\La)$ the bounded derived category of $\mmod\La$, which is equivalent to $\K^{b,-}(\proj\La)$, the homotopy category of right bounded complexes of finitely generated projectives over $\La$ with bounded cohomology. We denote the category of bounded complexes of finitely generated projectives over $\La$ by $\mathrm{C}^b(\proj\La)$, and the homotopy category of bounded complexes of finitely generated projectives over $\La$ by $\K^b(\proj\La)$. For $n,m\in \ZZ$, we set
\begin{equation*}
    \begin{split}
        \D^{\leq n}(\mmod\La)&:=\{X\in \D^b(\mmod\La)\mid \Ho^{>n}(X)=0\}\\
        \D^{\geq m}(\mmod\La)&:=\{X\in \D^b(\mmod\La)\mid \Ho^{<m}(X)=0\}\\
        \D^{[m,n]}(\mmod\La)&:=\D^{\leq n}(\mmod\La)\cap\D^{\geq m}(\mmod\La),
    \end{split}
\end{equation*}
and call them the \emph{truncated derived categories} of $\La$.
\medskip
Similarly, we define
\begin{equation*}
    \begin{split}
    \K^{\geq m}(\proj\La)&:= \{X\in \K^{b,-}(\proj\La)\mid \exists \ Y\in \K^{b,-}(\proj\La) \mbox{ with } X\cong Y \mbox{ and } Y^{<m}=0\}\\
    \K^{\leq m}(\proj\La)&:= \{X\in \K^{b,-}(\proj\La)\mid \exists \ Y\in \K^{b,-}(\proj\La) \mbox{ with } X\cong Y \mbox{ and } Y^{>m}=0\}\\
        &\K^{[m,n]}(\proj\La):= \K^{\geq m}(\proj\La)\cap K^{\leq n}(\proj\La),
    \end{split}
\end{equation*}
and call them the \emph{truncated homotopy categories} of $\La$. Being extension closed subcategories of triangulated categories, both $\D^{[m,n]}(\mmod\La)$ and $\K^{[m,n]}(\proj\La)$ inherit extriangulated structures in the sense of \cite{NP}.

We recall that for any $d\in \ZZ$, the natural $t$-structure \cite{BBDG} $$(\D^{\leq d}(\mmod\La),\D^{\geq d+1}(\mmod\La))$$ on $\D^b(\mmod\La)$ induces truncation functors $\tau_{\leq d}$ and $\tau_{\geq d+1}$ defined as follows. 

Let $A^\bullet=(\cdots\to A^{-2}\xrightarrow{\delta^{-2}} A^{-1}\xrightarrow{\delta^{-1}} A^0\xrightarrow{\delta^{0}} A^1\xrightarrow{\delta^{1}} A^2\to\cdots)\in\D^b(\mmod\La)$. Then

\begin{equation*}
    \begin{split}
        \tau_{\leq d}A^{\bullet}&:=(\cdots\to A^{d-2}\to A^{d-1}\to\Ker{\delta^d}\to 0\to 0\to\cdots)\\
        \tau_{\geq d+1}A^{\bullet}&:=(\cdots\to 0\to 0\to A^{d}/\Ker{\delta^d}\to A^{d+1}\to A^{d+2}\to\cdots)
    \end{split}
\end{equation*}
and we have a triangle $\tau_{\leq d}A^{\bullet}\to A^\bullet\to \tau_{\geq d+1}A^{\bullet}$ (we will not draw the last arrow of a distinguished triangle).

Similarly, the natural co-$t$-structure (\cite{P,B}\cite[\S~3.4]{KY}) $$(\K^{\geq d}(\proj\La),\K^{\leq d}(\proj\La))$$ on $\K^b(\proj\La)$ induces the hard truncations $t_{\leq d}$ and $t_{\geq d-1}$ defined below.

Let $P^\bullet=(\cdots\to P^{-2}\xrightarrow{\delta^{-2}} P^{-1}\xrightarrow{\delta^{-1}} P^0\xrightarrow{\delta^{0}} P^1\xrightarrow{\delta^{1}}P^2\to\cdots)\in \mathrm{C}^b(\proj\La)$. Then
\begin{equation*}
    \begin{split}
        t_{\leq d-1}P^{\bullet}&:=(\cdots\to P^{d-2}\to P^{d-1}\to 0\to 0\to 0\to\cdots)\in \mathrm{C}^b(\proj\La)\\
        t_{\geq d}P^{\bullet}&:=(\cdots\to 0\to 0\to 0\to P^{d}\to P^{d+1}\to\cdots) \in \mathrm{C}^b(\proj\La),
    \end{split}
\end{equation*}
and we get a triangle $t_{\geq d}P^{\bullet}\to P^\bullet\to t_{\leq d-1}P^{\bullet}$ in $\K^b(\proj\La)$.

We now recall the definition of silting objects in $\K^b(\proj\La)$.

\begin{definition}
    Let $M\in \K^b(\proj\La)$. 
    \begin{enumerate}
        \item $M$ is called a \emph{silting object} if 
        \begin{enumerate}
        \item $\Hom(M,\Sigma^iM)=0$ for all $i>0$;
        \item $\thick M=\K^b(\proj\La)$, where $\thick M$ is the smallest full triangulated subcategory of $\K^b(\proj\La)$ containing $M$, and closed under summands and isomorphism.
        \end{enumerate}
        \item A silting object is called $d$\emph{-term} if it is isomorphic to a complex of projectives concentrated in degrees $\{-d+1,\cdots,0\}$.
    \end{enumerate}
\end{definition}
We denote by $\silt \La$ the set of isomorphism classes of basic silting objects in $\K^b(\proj\La)$. This can be equipped with a poset structure $\leq$ as follows (\cite[Theorem~2.11]{AI}).
$$P \leq Q :\iff  \Hom (Q, \Sigma^iP) = 0 \ \forall\ i>0$$

We will denote the subposet of $\silt\La$ of $d$-term silting objects by $d$-$\silt\La$.

\section{Bijection between torsion and cotorsion pairs}
\subsection{Truncation functors}
The goal of this section is to generalize the relation between $K^{[-1,0]}(\proj\La)$ and $\mmod\La$ to higher truncations of the homotopy category. 
\medskip
Fix $d\geq 2$. Then we note that the truncation functor $\tau_{\geq -d+2}:\D^b(\mmod\La)\to\D^b(\mmod\La)$ defined in \S~\ref{not} gives a restriction $$\tau_{\geq -d+2}:\K^{[-d+1,0]}(\proj\La)\to \D^{[-d+2,0]}(\mmod\La).$$ The following proposition recovers, for $d=2$, some known properties of the cohomology functor $\Ho^0:\K^{[-1,0]}(\proj\La)\to \mmod\La$.

\begin{proposition}\label{equiv}
    The truncation functor $$\tau_{\geq -d+2}:\K^{[-d+1,0]}(\proj\La)\to \D^{[-d+2,0]}(\mmod\La)$$ induces an equivalence between the categories $\frac{\K^{[-d+1,0]}(\proj\La)}{\add\La[d-1]}$ and $\D^{[-d+2,0]}(\mmod\La)$.
\end{proposition}
\begin{proof} We divide the proof into three steps. 
    \medskip
    \textbf{Step 1.} The functor $\tau_{\geq -d+2}$ is essentially surjective on $\D^{[-d+2,0]}(\mmod\La)$: Let $M\in \D^{[-d+2,0]}(\mmod\La)$. Then, using the fact that the inclusion of $\K^{b,-}(\proj\La)$ into $\D^b(\mmod\La)$ is an equivalence, $M$ is isomorphic to a complex of projectives bounded above, say, $(\cdots\to P^{-d+1}\to\cdots P^{-1}\to P^0\to 0)$. Then $P'=(\cdots\to0\to P^{-d+1}\to\cdots \to P^0\to0)\in\K^{[-d+1,0]}(\proj\La)$ and the following map gives a quasi-isomorphism between $M$ and $\tau_{\geq -d+2}(P')$.

    \begin{center}
    \begin{tikzcd}
    \cdots \arrow[r] & P^{-d} \arrow[r, "{\delta^{-d}}"] \arrow[d] & P^{-d+1} \arrow[d, "\pi"] \arrow[r, "\delta^{-d+1}"]          & P^{-d+2} \arrow[r] \arrow[d, "id"] & \cdots \arrow[r] & P^0 \arrow[r] \arrow[d, "id"] & 0 \\
    \cdots \arrow[r] & 0 \arrow[r]                                     & \frac{P^{-d+1}}{\Ker {\delta^{-d+1}}} \arrow[r, "\overline{\delta^{-d+1}}"'] & P^{-d+2} \arrow[r]                 & \cdots \arrow[r] & P^0 \arrow[r]                 & 0
    \end{tikzcd}
    \end{center}

    \medskip
    \textbf{Step 2.} The functor $\tau_{\geq -d+2}$ is full: Let $X,Y\in K^{[-d+1,0]}(\proj\La)$ and $f:\tau_{\geq -d+2}X\to \tau_{\geq -d+2}Y\in \D^{[-d+2,0]}(\mmod\La)$. By definition, $$\tau_{\geq -d+2}X=0\to \frac{X^{-d+1}}{\Ker{\delta_X^{-d+1}}}\xrightarrow{\overline{\delta_X^{-d+1}}}X^{-d+2}\to\cdots\to X^0\to 0$$ and $$\tau_{\geq -d+2}Y=0\to \frac{Y^{-d+1}}{\Ker{\delta_Y^{-d+1}}}\xrightarrow{\overline{\delta_Y^{-d+1}}}Y^{-d+2}\to\cdots\to Y^0\to 0.$$ Let $\cdots \to P^{-d-1}\to P^{-d}\to X^{-d+1}\to 0$ and $\cdots \to Q^{-d-1}\to Q^{-d}\to Y^{-d+1}\to 0$ be projective resolutions of $\frac{X^{-d+1}}{\Ker{\delta_X^{-d+1}}}$ and $\frac{Y^{-d+1}}{\Ker{\delta_Y^{-d+1}}}$ respectively. Set $X'=\cdots\to P^{-d}\to X^{-d+1}\to X^{-d+2}\to\cdots\to X^0\to 0$ and $Y'=\cdots\to Q^{-d}\to Y^{-d+1}\to Y^{-d+2}\to\cdots\to Y^0\to 0$. Then there exist canonical quasi-isomorphisms $i_X: X'\to \tau_{\geq -d+2}X $ and $i_Y:Y'\to \tau_{\geq -d+2}Y$. Using the equivalence between $\K^{b,-}(\proj\La)$ and $\D^b(\mmod\La)$ again, we get that $g:=(i_Y)^{-1}fi_X\in \K^{b,-}(\proj\La)$. Let $g': X\to Y$ be the map shown below. 
    \begin{center}
        \begin{tikzcd}
        0 \arrow[r] & X^{-d+1} \arrow[d, "g^{-d+1}"] \arrow[r] & X^{-d+2} \arrow[r] \arrow[d, "g^{-d+2}"] & \cdots \arrow[r] & X^0 \arrow[r] \arrow[d, "g^0"] & 0 \\0 \arrow[r] & Y^{-d+1} \arrow[r]                     & Y^{-d+2} \arrow[r]                       & \cdots \arrow[r] & Y^0 \arrow[r]                  & 0
        \end{tikzcd}
    \end{center}
    Then $\tau_{\geq -d+2}(g')=\tau_{\geq -d+2}(g)
    =\tau_{\geq -d+2}(f)=f$. 

    \medskip
    
    \textbf{Step 3.} The kernel of $\tau_{\geq -d+2}$ is the set of maps factoring through $\add\La[d-1]$: Clearly, if a map factors through $\add\La[d-1]$, then it lies in $\Ker{\tau_{\geq -d+2}}$. Let $f:X\to Y\in \Ker{\tau_{\geq -d+2}}$. As before, we can define $X',Y'\in \K^{b,-}(\proj\La)$ with canonical quasi-isomorphisms $i_X: X'\to \tau_{\geq -d+2}X $ and $i_Y:Y'\to \tau_{\geq -d+2}Y$. Since a morphism of modules can be lifted to a morphism of their projective resolutions, we can find maps $g^i:P^i\to Q^i$ for $i\leq -d$ such that the following is a morphism $h$ of complexes.
    \begin{center}
        \begin{tikzcd}
        \cdots \arrow[r] & P^{-d} \arrow[r] \arrow[d, "g^{-d}"] & X^{-d+1} \arrow[d, "f^{-d+1}"] \arrow[r] & X^{-d+2} \arrow[r] \arrow[d, "f^{-d+2}"] & \cdots \arrow[r] & X^0 \arrow[r] \arrow[d, "f^0"] & 0 \\
        \cdots \arrow[r] & Q^{-d} \arrow[r]                       & Y^{-d+1} \arrow[r]                     & Y^{-d+2} \arrow[r]                       & \cdots \arrow[r] & Y^0 \arrow[r]& 0\end{tikzcd}
    \end{center}
    Again it can be checked that $\tau_{\geq -d+2}(f)\circ i_X=i_Y\circ h$, which implies that $h=0$. Thus, there exist maps $k^i:X^i\to Y^{i-1}$ such that the homotopy relation holds. 
    \begin{center}
        \begin{tikzcd}
        \cdots \arrow[r] & P^{-d} \arrow[r] \arrow[d, "g^{-d}"] & X^{-d+1} \arrow[d, "f^{-d+1}"] \arrow[r] \arrow[ld, "k^{-d+1}", dashed] & X^{-d+2} \arrow[r] \arrow[d, "f^{-d+2}"] \arrow[ld, "k^{-d+2}", dashed] & \cdots \arrow[r] \arrow[ld, dashed] & X^0 \arrow[r] \arrow[d, "f^0"] \arrow[ld, dashed] & 0 \\\cdots \arrow[r] & Q^{-d} \arrow[r, "\delta_Y^{-d}"']                       & Y^{-d+1} \arrow[r]                                                  & Y^{-d+2} \arrow[r]                                                      & \cdots \arrow[r]            & Y^0 \arrow[r]                             & 0
        \end{tikzcd}
    \end{center}
    Consider the following two maps. 
    \begin{center}
        \begin{tikzcd}
        0 \arrow[r] & X^{-d+1} \arrow[d, "id"] \arrow[r] & X^{-d+2} \arrow[r] \arrow[d] & \cdots \arrow[r] & X^0 \arrow[r] \arrow[d] & 0 \\
        0 \arrow[r] & X^{-d+1} \arrow[r]   \arrow[d, "\delta^{-d}_Y\circ k^{-d+1}"']              & 0 \arrow[r]  \arrow[d]                & \cdots \arrow[r] & 0 \arrow[r]   \arrow[d]          & 0
        \\0 \arrow[r] & Y^{-d+1} \arrow[r]                                           & Y^{-d+2} \arrow[r]    & \cdots \arrow[r] & Y^0 \arrow[r]         & 0
        \end{tikzcd}
    \end{center}
    The diagram below shows that the composition of the above two maps is equal to $f$ in the homotopy category. 

    \begin{center}
    \begin{tikzcd}
    0 \arrow[r] & X^{-d+1} \arrow[d, "f^{-d+1}-\delta^{-d}_Y\circ k^{-d+1}"'] \arrow[r] & X^{-d+2} \arrow[r] \arrow[d, "f^{-d+2}"] \arrow[ld, "k^{-d+2}", dashed] & \cdots \arrow[r] \arrow[ld, dashed] & X^0 \arrow[r] \arrow[d, "f^0"] \arrow[ld, dashed] & 0 \\
    0 \arrow[r] & Y^{-d+1} \arrow[r]& Y^{-d+2} \arrow[r]  & \cdots \arrow[r]    & Y^0 \arrow[r] & 0
    \end{tikzcd}    
    \end{center}
    Thus, $f$ factors through $\add\La[d-1]$, and we get an induced equivalence $$\tau_{\geq -d+2}: \frac{\K^{[-d+1,0]}(\proj\La)}{\add\La[d-1]} \to \D^{[-d+2,0]}(\mmod\La).$$
    
\end{proof}

Our next goal is to show that we can lift conflations in $\D^{[-d+2,0]}(\mmod\La)$ to conflations in $\K^{[-d+1,0]}(\proj\La)$ across this equivalence (we talk about the conflations coming from the respective extriangulated structures). We will need the following lemma about inflations in $\K^{[-d+1,0]}(\proj\La)$.
\begin{lemma}\label{split}
    Let $f: X\to Y$ be a morphism in $\K^{[-d+1,0]}(\proj\La)$ such that $f^{-d+1}$ is a split mono. Then $f$ is an inflation.
\end{lemma}
\begin{proof}
    We need to show that $C(f)\in \K^{[-d+1,0]}(\proj\La)$. By definition, $C(f)$ is given as $$0\to X^{-d+1}\xrightarrow{\begin{bmatrix}
        \delta_X^{-d+1}\\ f^{-d+1}
    \end{bmatrix}}X^{-d+2}\oplus Y^{-d+1}\to \cdots X^0\oplus Y_1\to Y^0.$$ Since $f^{-d+1}$ is a split mono, so is $\begin{bmatrix}
        \delta_X^{-d+1}\\ f^{-d+1}
    \end{bmatrix}$. Thus, $C(f)$ is isomorphic to $$0\to X^{-d+1}\xrightarrow{\begin{bmatrix}
        1_{X^{-d+1}}\\0
    \end{bmatrix}} X^{-d+1}\oplus Z\to \cdots X^0\oplus Y_1\to Y^0$$ for some $Z\in\proj\La$. Thus $C(f)=$
    \begin{center}
        \begin{tikzcd}
      & 0 \arrow[r] & X^{-d+1} \arrow[r, "1_{X^{-d+1}}"] & X^{-d+1} \arrow[r]         & \cdots \arrow[r] & 0 \arrow[r]             & 0   \\
 &             &                                    & \oplus                     &                  &                         &     \\
      & 0 \arrow[r] & 0 \arrow[r]                        &  Z \arrow[r] & \cdots \arrow[r] & X^0\oplus Y^1 \arrow[r] & Y^0
\end{tikzcd}
    \end{center}Hence, it is quasi-isomorphic to $$ (Z \to \cdots \to  X^0\oplus Y^1 \to  Y^0)\in \K^{[-d+1,0]}(\proj\La).$$
\end{proof}

\begin{proposition}\label{liftconf}
    Let $\epsilon: X\xrightarrow{f} Y\xrightarrow{g} Z$ be a conflation in $\D^{[-d+2,0]}(\mmod\La)$. Then there exists a conflation $\epsilon': X'\xrightarrow{f'}Y'\xrightarrow{g'}Z'$ in $\K^{[-d+1,0]}(\proj\La)$ such that $\tau_{\geq -d+2}\epsilon'\cong \epsilon$.
\end{proposition}
\begin{proof}
    Using the equivalence between $\K^{b,-}(\proj\La)$ and $\D^b(\mmod\La)$, we can assume that $f$ is of the following form:
    \begin{center}
        \begin{tikzcd}
        \cdots \arrow[r] & P^{-d+1} \arrow[d, "f^{-d+1}"'] \arrow[r] & P^{-d+2} \arrow[r] \arrow[d, "f^{-d+2}"] & \cdots \arrow[r] & P^0 \arrow[r] \arrow[d, "f^0"] & 0 \\
        \cdots \arrow[r] & Q^{-d+1} \arrow[r]                      & Q^{-d+2} \arrow[r]                       & \cdots \arrow[r] & Q^0 \arrow[r]                  & 0
        \end{tikzcd}
    \end{center}
    This implies that $\epsilon\cong X\xrightarrow{f}Y\xrightarrow{g} C(f)$. Set $$X'=t_{\geq -d+1}(X)= 0\to P^{-d+1}\to P^{-d+2}\to\cdots\to P^0\to0,$$ and $$Y''=t_{\geq -d+1}(Y)= 0\to Q^{-d+1}\to Q^{-d+2}\to\cdots\to Q^0\to 0,$$ and $f''=t_{\geq -d+1}(f)$. Set $Y':=Y''\oplus P^{-d+1}[d-1]$. Let $i: X'\to P^{-d+1}[d-1]$ be the map that is identity in degree $-d+1$, and $0$ in the others. Then the map $f':=\begin{bmatrix}
        f''\\ i
    \end{bmatrix}: X'\to Y'$ is an inflation in $\K^{[-d+1,0]}(\proj\La)$ by Lemma \ref{split} because it is a split mono in degree $-d+1$. We claim that $\tau_{\geq -d+2}(X'\xrightarrow{f'}Y'\xrightarrow{g'}C(f'))\cong \epsilon$. 

    Using the previous proposition, we know that $\tau_{\geq -d+2}(f')\cong \tau_{\geq -d+2}(f'')\cong f$. By definition, 
    \begin{align*}
        \tau_{\geq -d+2}C(f') &= 0\to \frac{P^{-d+2}\oplus Q^{-d+1}\oplus P^{-d+1}}{\Ker{\delta_{C(f)}^{-d+1}}\oplus P^{-d+1}}\to P^{-d+3}\oplus Q^{-d+2}\to\cdots Q^0\to 0\\ &\cong 0\to \frac{P^{-d+2}\oplus Q^{-d+1}}{\Ker{\delta_{C(f)}^{-d+1}}}\to P^{-d+3}\oplus Q^{-d+2}\to\cdots Q^0\to 0,
    \end{align*}
    where $\delta_{C(f)}^{-d+1}$ denotes the $-d+1$ boundary map of $C(f)$. Since $C(f)\in \D^{[-d+2,0]}(\mmod\La)$, the following map gives a quasi-isomorphism between $\tau_{\geq -d+2}C(f')$ and $C(f)$.
    \begin{center}
        \begin{tikzcd}
        \cdots \arrow[r] & P^{-d+1}\oplus Q^{-d} \arrow[d] \arrow[r] & P^{-d+2}\oplus Q^{-d+1} \arrow[r] \arrow[d, "\pi"] & \cdots \arrow[r] & Q^0 \arrow[r] \arrow[d, "id"] & 0 \\
        \cdots \arrow[r] & 0 \arrow[r] & \frac{P^{-d+2}\oplus Q^{-d+1}}{\Ker {\delta_{C(f)}^{-d+1}}} \arrow[r]& \cdots \arrow[r] & Q^0 \arrow[r] & 0
        \end{tikzcd}
    \end{center}
    Denoting the above map by $i_{C(f')}$, it can be easily checked that the following diagram commutes, 
    \begin{center}
        \begin{tikzcd}
        X \arrow[r, "f"] \arrow[d, "i_{X'}"']              & Y \arrow[r, "g"] \arrow[d, "i_{Y'}"']              & C(f) \arrow[d, "i_{C(f')}"] \\
        \tau_{\geq-d+2}X' \arrow[r, "\tau_{\geq -d+2}f'"'] & \tau_{\geq-d+2}Y' \arrow[r, "\tau_{\geq -d+2}g'"'] & \tau_{\geq-d+2}C(f')       
        \end{tikzcd}
    \end{center} where $i_{X'}$ and $i_{Y'}$ are the canonical quasi-isomorphisms defined in the previous proposition. 
    
    Setting $\epsilon'=X'\xrightarrow{f'}Y'\to C(f')$, we get that $\tau_{\geq -d+1}\epsilon'\cong \epsilon$.
\end{proof}
The following corollary gives the dual of the above two propositions for the homotopy category of injectives. 
\begin{corollary}\label{dual}
\begin{enumerate}
    \item The functor $\tau_{\leq 0}$ induces an equivalence between the categories $\frac{\K^{[-d+2,1]}(\inj\La)}{\inj\La[-1]}$ and $\D^{[-d+2,0]}(\mmod\La)$.
    \item Given a conflation $\epsilon: X\xrightarrow{f} Y\xrightarrow{g} Z$ in $\D^{[-d+2,0]}(\mmod\La)$, there exists a conflation $\epsilon': X'\xrightarrow{f'}Y'\xrightarrow{g'}Z'$ in $\K^{[-d+2,1]}(\inj\La)$ such that $\tau_{\leq 0}\epsilon'\cong \epsilon$.
\end{enumerate}
\end{corollary}
Our next goal is to characterize the extensions between any two objects in $\K^{[-d+1,0]}(\proj\La)$ in terms of the morphisms between their images under the truncation functors. This will be used in the next section to give a correspondence between cotorsion pairs in $\K^{[-d+1,0]}(\proj\La)$ and torsion pairs in $\D^{[-d+2,0]}(\mmod\La)$. 

\begin{theorem}\label{Happel}\cite[\S~I.4.6]{H1}
    Let $X,Y\in \K^b(\proj\La)$. Then $$\Hom_{\K^b(\proj\La)}(X,Y)\cong \mathrm{D}\Hom_{\D^b(\mmod\La)}(Y,\nu X),$$ where $\nu$ denotes the Nakayama functor. 
\end{theorem}
In the following lemma, we view $\K^{[-d+1,0]}(\proj\La)$ as an extriangulated category with $\Eb(X,Y)$ denoting the group of extensions from $X$ to $Y$. We will use this notation again in \S~\ref{sec:cotors}.
\begin{lemma}\label{Extlem}
    Let $X,Y\in\K^{[-d+1,0]}(\proj\La)$. Then 
    $$\Eb(X,Y)\cong \mathrm{D}\Hom_{\D^b(\mmod\La)}(\tau_{\geq -d+2}Y, \Sigma^{-1}\tau_{\leq -1}\nu X).$$
\end{lemma}
\begin{proof}
Using Theorem \ref{Happel}, we get that $$\Eb(X,Y)\cong\mathrm{D}\Hom_{\K^b(\proj\La)}(X,\Sigma Y)\cong \mathrm{D}\Hom_{\D^b(\mmod\La)}(\Sigma Y,\nu X).$$

The triangle $\tau_{\leq -d}\Sigma Y\to \Sigma Y\to \tau_{\geq -d+1}\Sigma Y$ gives the exact sequence $$\Eb^{-1}(\tau_{\leq -d}\Sigma Y, \nu X)\to \Hom(\tau_{\geq -d+1}\Sigma Y,\nu X)\to \Hom(\Sigma Y,\nu X)\to \Hom(\tau_{\leq -d}\Sigma Y,\nu X).$$ By comparing the support of the cohomology of $\tau_{\leq -d}\Sigma Y$ and $\nu X$, we see that $\Eb^{-1}(\tau_{\leq -d}\Sigma Y, \nu X)$and $\Hom(\tau_{\leq -d}\Sigma Y,\nu X)$ vanish. Thus, $$\Hom(\tau_{\geq -d+1}\Sigma Y,\nu X)\cong \Hom(\Sigma Y,\nu X).$$

Now, the triangle $\tau_{\leq -1}\nu X\to \nu X\to \tau_{\geq 0}\nu X$ gives the exact sequence 
\begin{equation*}
    \begin{tikzcd}[column sep=huge]
        \Eb^{-1}(\tau_{\geq -d+1}\Sigma Y,\tau_{\geq 0}\nu X) \arrow[r] \arrow[d, phantom, ""{coordinate, name=Z}]& \Hom(\tau_{\geq -d+1}\Sigma Y,\tau_{\leq -1}\nu X) \arrow[dl, rounded corners,
        to path={ -- ([xshift=2ex]\tikztostart.east)
            |- (Z) [near end]\tikztonodes
            -| ([xshift=-2ex]\tikztotarget.west)
            -- (\tikztotarget)}] \\
        \Hom(\tau_{\geq -d+1}\Sigma Y,\nu X) \arrow[r] & \Hom(\tau_{\geq -d+1}\Sigma Y, \tau_{\geq 0}\nu X).
    \end{tikzcd}
\end{equation*}
Again, we see that $\Eb^{-1}(\tau_{\geq -d+1}\Sigma Y,\tau_{\geq 0}\nu X)$ and $\Hom(\tau_{\geq -d+1}\Sigma Y, \tau_{\geq 0}\nu X)$ vanish. Thus, 
\begin{equation*}
    \begin{split}
        \Hom(\tau_{\geq -d+1}\Sigma Y,\nu X)&\cong\Hom(\tau_{\geq -d+1}\Sigma Y,\tau_{\leq -1}\nu X)\\ &\cong \Hom(\Sigma\tau_{\geq -d+2}Y, \tau_{\leq -1}\nu X)\\ &\cong \Hom(\tau_{\geq -d+2}Y, \Sigma^{-1}\tau_{\leq -1}\nu X).
    \end{split}
\end{equation*}
\end{proof}

 We need the following technical lemmas for the subsequent sections.

\begin{lemma}\label{techlem}
    Let $f:X\to Y\in \K^{[-d+1,0]}(\proj\La)$. Then $$\tau_{\geq -d+2}C(f)\cong \tau_{\geq -d+2}C(\tau_{\geq -d+2}f).$$
\end{lemma}
\begin{proof}
    Suppose $f$ is of the following form. 
    \begin{center}
        \begin{tikzcd}
        X^{-d+1} \arrow[r, "\delta_X^{-d+1}"] \arrow[d, "f^{-d+1}"] & \cdots \arrow[r] & X^{-1} \arrow[r, "\delta_X^{-1}"] \arrow[d, "f^{-1}"] & X^0 \arrow[d, "f^0"] \\Y^{-d+1} \arrow[r, "\delta_Y^{-d+1}"'] & \cdots \arrow[r] & Y^{-1} \arrow[r, "\delta_Y^{-1}"']                    & Y^0  \end{tikzcd}
    \end{center}
    Then, by definition, $\tau_{\geq -d+2}C(f)$ is 
    \begin{center}
        \begin{tikzcd}
            \frac{X^{-d+2}\oplus Y^{-d+1}}{\Ker {\begin{bmatrix}
                -\delta_X^{-d+2} & 0\\f^{-d+2}&\delta_Y^{-d+1}
            \end{bmatrix}}} \arrow[r]& X^{-d+3}\oplus Y^{-d+2}\arrow[r]&\cdots \arrow[r] & X^0\oplus Y^{-1} \arrow[r]& Y^0 .
        \end{tikzcd}
    \end{center}
    On the other hand, $\tau_{\geq -d+2}C(\tau_{\geq -d+2}f)$ is given as
    \begin{center}
        \begin{tikzcd}
            \frac{X^{-d+2}\oplus \frac{Y^{-d+1}}{\Ker {\delta_{Y}^{-d+1}}}}{\Ker {\begin{bmatrix}
                -\delta_X^{-d+2} & 0\\f^{-d+2}&\overline{\delta_Y^{-d+1}}
            \end{bmatrix}}} \arrow[r]& X^{-d+3}\oplus Y^{-d+2}\arrow[r]& \cdots \arrow[r] & X^0\oplus Y^{-1} \arrow[r]& Y^0 .
        \end{tikzcd}
    \end{center}
    The following map gives an isomorphism between the two since the map ${\overline{(x,y)}\mapsto\overline{(x,\overline{y})}}$ is an isomorphism.
    \begin{center}
        \begin{tikzcd}
        \frac{X^{-d+2}\oplus Y^{-d+1}}{\Ker {\begin{bmatrix}                 -\delta_X^{-d+2} & 0\\f^{-d+2}&\delta_Y^{-d+1}             \end{bmatrix}}} \arrow[r] \arrow[d, "{\overline{(x,y)}\mapsto\overline{(x,\overline{y})}}"] & X^{-d+3}\oplus Y^{-d+2}\arrow[r]\arrow[d, "id"]& \cdots \arrow[r] & X^0\oplus Y^{-1} \arrow[r] \arrow[d, "id"] & Y^0 \arrow[d, "id"] \\
        \frac{X^{-d+2}\oplus \frac{Y^{-d+1}}{\Ker {\delta_{Y}^{-d+1}}}}{\Ker {\begin{bmatrix} -\delta_X^{-d+2} & 0\\f^{-d+2}&\overline{\delta_Y^{-d+1}}  \end{bmatrix}}} \arrow[r]& X^{-d+3}\oplus Y^{-d+2}\arrow[r]& \cdots \arrow[r] & X^0\oplus Y^{-1} \arrow[r]   & Y^0                
        \end{tikzcd}
    \end{center}
\end{proof}
Recall that a full subcategory $\X$ of a category $\C$ is called \emph{contravariantly finite} (resp. \emph{covariantly finite}) if for all $C\in\C$, there exists $f_C: X_C\to C$ (resp. $g_C: C\to X_C$) with $X_C\in\X$ such that for all $X\in \X$, $\Hom(X,X_C)\xrightarrow{\Hom(X,f_C)} \Hom(X,C)$ (resp. $\Hom(X_C,X)\xrightarrow{\Hom(g_C,X)}\Hom(C,X)$) is surjective.
\begin{lemma}\label{approximations}
    Let $\Y\subseteq \K^{[-d+1,0]}(\proj\La)$ be an additive subcategory containing $\add\La[d-1]$. Then $\Y$ is covariantly (resp. contravariantly) finite in $\K^{[-d+1,0]}(\proj\La)$ if and only if $\tau_{\geq -d+2}\Y$ is covariantly (resp. contravariantly) finite in $\D^{[-d+2,0]}(\proj\La)$.
\end{lemma}
\begin{proof}
    We will only prove the case for covariant finiteness. The case for contravariant finiteness will be dual.
    \medskip
    $(\implies)$ Follows from the fact that the functor $\tau_{\geq -d+2}$ is full and essentially surjective. 
    \medskip
    $(\impliedby)$ Let $Z\in \K^{[-d+1,0]}(\proj\La)$ and $f':\tau_{\geq -d+2}Z\to Y'$ be a left $\tau_{\geq -d+2}\Y$-approximation of $\tau_{\geq -d+2}Z$. Since $\tau_{\geq -d+2}$ is full and essentially surjective, there exists $f:Z\to Y\in \K^{[-d+1,0]}(\proj\La)$ such that $\tau_{\geq -d+2}f=f'$. Let $g: Z\to P$ be a left $\add\La[d-1]$-approximation of $Z$, which exists since $\add\La[d-1]$ is functorially finite. We claim that $\begin{bmatrix}
        f\\g
    \end{bmatrix}: Z\to Y\oplus P$ is a left $\Y$-approximation of $Z$. 

    Let $f_1: Z\to Y_1$ be a morphism with $Y_1\in \Y$. Since $f'$ is an approximation, there exists $g_1:Y\to Y_1$ such that $f'\tau_{\geq -d+2}g_1=\tau_{\geq -d+2}f_1$. Therefore, the map $g_1f-f_1$ factors through some $P'\in \add\La[d-1]$ as shown below. 
    \begin{center}
    \begin{tikzcd}
    Z \arrow[rr, "g_1f-f_1"] \arrow[rd, "h"'] &  & Y_1 \\
    & P' \arrow[ru, "k"'] &    
    \end{tikzcd}
    \end{center}
    Since $g:Z\to P$ is a left $\add\La[d-1]$-approximation of $Z$, there exists some $l:P\to P'$ such that $lg=h$. This gives the following commutative triangle and we are done.
    \begin{center}
    \begin{tikzcd}
    Z \arrow[rr, "\begin{bmatrix}f\\g\end{bmatrix}"] \arrow[rrd, "f_1"'] &  & Y\oplus P \arrow[d, "\begin{bmatrix}g_1 \  -kl\end{bmatrix}"] \\
    &  & Y_1                                                
    \end{tikzcd}
    \end{center}
\end{proof}
\begin{lemma}\label{inflation}
    For any morphism $f:A \to B\in \K^{[-d+1,0]}(\proj\La)$, there exists an inflation $f': A\to B'\in \K^{[-d+1,0]}(\proj\La)$ such that $\tau_{\geq -d+2}f\cong \tau_{\geq -d+2}f'$.
\end{lemma}
\begin{proof}
    Let $f':A\to A^{-d+1}[d-1]\oplus B$ be the following morphism. 
    \begin{center}
        \begin{tikzcd}
        A^{-d+1} \arrow[d, "\begin{bmatrix}id\\ f^{-d+1}\end{bmatrix}"'] \arrow[r, "\delta^{-d+1}_A"] & A^{-d+2} \arrow[d, "f^{-d+1}"] \arrow[r] & \cdots \arrow[d] \arrow[r] & A^{-1} \arrow[d, "f^{-1}"] \arrow[r] & A^0 \arrow[d, "f^0"] \\
        A^{-d+1}\oplus B^{-d+1} \arrow[r, "{[0 \ \ \delta_B^{-d+1}]}"']     & B^{-d+2} \arrow[r]& \cdots \arrow[r] & B^{-1} \arrow[r]& B^0
        \end{tikzcd}
    \end{center}
    Then, clearly, $\tau_{\geq -d+2}f\cong \tau_{\geq -d+2}f'$. Since $f'^{-d+1}$ is a split mono, using Lemma \ref{split}, we get that $f'$ is an inflation in $\K^{[-d+1,0]}(\proj\La)$. 
\end{proof}

\subsection{From cotorsion pairs to torsion pairs}\label{sec:cotors}
Let $(\C,\Eb,\s)$ be an extriangulated category equipped with a bivariant $\delta$-functor $(G^\bullet, \delta_\#, \delta^\#)$ such that $G^0=\Hom_\C$, $G^1=\Eb$, and for any $\Eb$-extension $\delta\in \Eb(C,A)$, $\delta^0_\#=A_\delta$ and $\delta^{\#,0}=B_\delta$. In particular, $(\C,\Eb,\s, G^{-1})$ is an extriangulated category with a negative first extension in the sense of \cite[Definition~2.3]{AET}. 

Henceforth, we will consider $\D^{[-d+2,0]}(\mmod\La)$ and $\K^{[-d+1,0]}(\proj\La)$ as extriangulated categories with bivariant $\delta$-functors $G^i(-,-):=\Eb^i(-,-)=\Hom(-,\Sigma^i-)$, where $\Sigma$ is the suspension functor of the triangulated categories $\D^b(\mmod\La)$ and $\K^{b}(\proj\La)$ respectively. 

The notion of cotorsion pairs was first defined in \cite{Sa} for abelian categories. These were then generalised to exact categories \cite{H2, St}, triangulated categories \cite{N}, and extriangulated categories \cite{NP}. We will use the following definitions in this work.  
\begin{definition}
    Let $\X,\Y$ be full subcategories of $\C$. 
    \begin{enumerate}
        \item The pair $(\X,\Y)$ is called a \emph{cotorsion pair} if 
        \begin{enumerate}
        \item $\Eb(x,\Y)=0$ if and only if $x\in\X$;
        \item $\Eb(\X,y)=0$ if and only if $y\in\Y$.
    \end{enumerate}
        \item  A cotorsion pair $(\X ,\Y)$ is called \emph{hereditary} if $G^k(\X,\Y)=0$ for all $k\geq 2$.
        \item A cotorsion pair $(\X ,\Y)$ is called \emph{complete} if for each $c \in \C$
        \begin{enumerate}
            \item[(Ca)] there exists a conflation $c \to y \to x$ with $x \in \X$ and $y \in \Y$;
            \item[(Cb)] there exists a conflation $y' \to x' \to c$ with $x' \in \X$ and $y' \in \Y$.
        \end{enumerate}
    \end{enumerate}
    \end{definition}

\begin{remark}\label{complete}
    Using \cite[Remark~4.4]{NP}, to check if a pair of full subcategories $(\X,\Y)$ is complete, it is enough to check the weaker conditions $\Eb(\X,\Y)=0$, (Ca), and (Cb).
\end{remark}
\begin{remark}
The definition of cotorsion pairs for extriangulated categories was first introduced in \cite{NP}, and that of hereditary cotorsion pairs in \cite{AT}. However, we would like to point out that cotorsion pairs in the sense of \cite{NP, AT} are complete cotorsion pairs in our sense. Moreover, the authors in \cite{AT} also define the notion of $s$-cotorsion pairs which, in $\K^{[-d+1,0]}(\proj\La)$, turns out to be equivalent to the notion of hereditary cotorsion pairs (Lemma \ref{s-implies-h}). 
\end{remark}

Following \cite{Sa}, for a cotorsion pair $(\X,\Y)$, we will call $\Y$ a \emph{cotorsion class}, and $\X$ a \emph{cotorsion-free class}. We will call a cotorsion class \emph{hereditary} (resp. \emph{complete}) if the associated cotorsion pair is hereditary (resp. complete). We will denote the poset of cotorsion classes in $\C$ under inclusion by $\cotors \C$. This is isomorphic to the poset of cotorsion pairs in $\C$ under the order $$(X_1,Y_1)\preceq(X_2,Y_2):\iff X_1\supseteq X_2\iff Y_1\subseteq Y_2.$$ Furthermore, we will denote the subposets of hereditary, complete, and complete and hereditary cotorsion classes/pairs by $\hcotors\C$, $\ccotors\C$, and $\chcotors\C$, respectively.
    
\begin{lemma}\label{s-implies-h}
    Let $(\X,\Y)$ be a cotorsion pair in $\K^{[-d+1,0]}(\proj\La)$. Then the following are equivalent.
        \begin{enumerate}
        \item $(\X,\Y)$ is a hereditary cotorsion pair.
        \item $\Eb^2(\X,\Y)=0$.
        \item $\X$ is closed under cocones. 
        \item $\Y$ is closed under cones. 
        \end{enumerate}
    \end{lemma}
    \begin{proof}
        ($1\implies 2$) Clear. 
        \newline
        ($2\implies 3$, $2\implies 4$) The proof of \cite[Lemma~3.2]{AT} does not use the completeness axioms assumed by the authors.
        \newline
        ($3\implies 1$) Let $X\in\X$. Then there is a conflation $\Sigma^{-1}t_{\leq -1}X\to t_{\geq 0}X\to X$ with $t_{\geq 0}X,X\in\X$. Since $\X$ is closed under cocones, $t_{\leq 0}\Sigma^{-1}X\cong\Sigma^{-1}t_{\leq -1}X\in\X$. Repeating the above argument, we conclude that $t_{\leq 0}\Sigma^{-k}X\in\X$ for all $k\geq 1$.
        
        Using the triangle $t_{\geq 1}\Sigma^{-k}X\to\Sigma^{-k}X\to t_{\leq 0}\Sigma^{-k}X$, we get that $\Eb^{k+1}(X,\Y)\cong\Eb(\Sigma^{-k}X,\Y)=0$. 
        \newline
        ($4\implies 1$) Dual to the above proof.
    \end{proof}
Similar to cotorsion pairs, the notion of torsion pairs was first introduced for abelian categories in \cite{D}. This was then generalised to triangulated categories in \cite{BR}, where they were proven to be equivalent to $t$-structures. A weaker definition of torsion pairs in triangulated categories was used in \cite{IJ}, which we generalise here to extriangulated categories equipped with a $\delta$-functor.
\begin{definition}\label{torclass}
    Let $\T,\F$ be full subcategories of $\C$. 
    \begin{enumerate}
        \item The pair $(\T,\F)$ is called a \emph{torsion pair} if 
        \begin{enumerate}
        \item $\Hom(t,\F)=0$ if and only if $t\in\T$;
        \item $\Hom(\T,f)=0$ if and only if $f\in\F$.
        \end{enumerate}
        \item A torsion pair $(\T,\F)$ is called \emph{positive} if $G^{k}(\T,\F)=0$ for all $k\leq -1$.
        \item\cite[Definition~3.1]{AET} The pair $(\T,\F)$ is called an $s$\emph{-torsion pair} if 
        \begin{enumerate}
        \item[(Sa)] $\Hom(\T,\F)=0$;
        \item[(Sb)] $G^{-1}(\T,\F)=0$;
        \item[(Sc)]\label{Sc} For all $C\in \C$, there exists a conflation $T\to C\to F$ with $T\in \T$ and $F\in\F$.
        \end{enumerate}
    \end{enumerate}
\end{definition}
\begin{remark}
    Using \cite[Proposition~3.2]{AET}, we know that every $s$-torsion pair is a torsion pair.
\end{remark}
For a torsion pair $(\T,\F)$, we will call $\T$ a \emph{torsion class}, and $\F$ a \emph{torsion-free class}. We will call a torsion class \emph{positive} (resp. an $s$-torsion class) if the associated torsion pair is positive (resp. an $s$-torsion pair). We will denote the poset of torsion classes in $\C$ under inclusion by $\tors \C$. This is isomorphic to the poset of torsion pairs in $\C$ under the order 
$$(\T_1,\F_1)\preceq (\T_2,\F_2)\iff \T_1\subseteq \T_2\iff \F_1\supseteq \F_2.$$ We will denote the subposets of positive, contravariantly finite, covariantly finite, and functorially finite torsion classes/pairs by $\ptors\C$, $\contrtors\C$, $\covtors\C$, $\ftors\C$, respectively. The subposet of $s$-torsion pairs will be denoted by $\stors\C$. Finally, we will denote the intersection of $\xtors\C$ and $\ytors\C$ by $\xytors\C$, where $\mbox{x, y}\in\{\mbox{p, contr, cov, f, s}\}$.

\begin{lemma}\label{p-implies-s}
    Let $(\T,\F)$ be a torsion pair in $\D^{[-d+2,0]}(\mmod\La)$. Then the following are equivalent.
    \begin{enumerate}
        \item\label{p} $(\T,\F)$ is a positive torsion pair.
        \item\label{s} $\Eb^{-1}(\T,\F)=0$.
        \item For all morphisms $f\in\T$, $\tau_{\geq -d+2}C(f)\in\T$.
        \item For all morphisms $g\in \F$, $\tau_{\leq 0}\Sigma^{-1}C(g)\in\F$.
    \end{enumerate}
\end{lemma}
\begin{proof}
    $(1\implies 2)$ Clear. 
    \newline
    $(2\implies 3)$ Let $f:X\to Y\in\T$. For $F\in\F$, the triangle $X\to Y\to C(f)$ induces the exact sequence $$\Eb^{-1}(X,F)\to \Hom(C(f),F)\to \Hom(Y,F).$$ By assumption, $\Eb^{-1}(X,F)$ and $\Hom(Y,F)$ vanish, which gives that $\Hom(C(f),F)=0$. Using the triangle $\tau_{\leq -d+1}C(f)\to C(f)\to \tau_{\geq -d+2}C(f)$, we get the exact sequence $$\Hom(\tau_{\leq -d+1}C(f)[1],F)\to \Hom(\tau_{\geq -d+2}C(f),F)\to \Hom(C(f),F)=0.$$ Comparing the support of the cohomology of $\tau_{\leq -d+1}C(f)[1]$ and $F$, we get that $\Hom(\tau_{\leq -d+1}C(f)[1],F)=0$ which gives that $\Hom(\tau_{\geq -d+2}C(f),F)=0$. Thus, $\tau_{\geq -d+2}C(f)\in \prescript{\perp}{}{\F}=\T$.
    \newline
    $(2\implies 4)$ Dual to the above proof.
    \newline
    $(3\implies 1)$ Let $T\in\T$. Using the triangle $T\to 0\to \Sigma T$, we get that $\tau_{\geq -d+2}\Sigma T\in \T$. Thus, $\tau_{\geq -d+2}\Sigma (\tau_{\geq -d+2}\Sigma T)\in \T$. But this is the same as $\tau_{\geq -d+2}\Sigma^2 T$. Continuing in this way, we conclude that $\tau_{\geq -d+2}\Sigma^k T\in \T$ for all $k\geq 0$.

    Using the triangle $\tau_{\leq -d+1}\Sigma^{k}T\to\Sigma^{k}T\to \tau_{\geq -d+2}\Sigma^{k}T$, for all $F\in\F$, we get the exact sequence $$0=\Hom(\tau_{\geq -d+2}\Sigma^k T, F)\to \Hom(\Sigma^kT,F)\to \Hom(\tau_{\leq -d+1}\Sigma^{k}T,F).$$ Comparing the support of the cohomology of $\tau_{\leq -d+1}\Sigma^{k}T$ and $F$, we get that $\Hom(\tau_{\leq -d+1}\Sigma^{k}T,F)=0$. Thus $\Eb^{-k}(T,F)\cong \Hom(\Sigma^{k}T,F)=0$ for all $k\geq 0$, and $(\T,\F)$ is a positive torsion pair.
    \newline
    $(4\implies 1)$ Dual to the above proof.
\end{proof}
\begin{corollary}\label{pconts}
    The class of $s$-torsion classes in $\D^{[-d+2,0]}(\mmod\La)$ is contained in the class of positive torsion classes in $\D^{[-d+2,0]}(\mmod\La)$.
\end{corollary}
\begin{proof}
    This follows from (\ref{s}) $\implies$ (\ref{p}) in Lemma \ref{p-implies-s}.
\end{proof}

The next theorem provides us with a way to link cotorsion pairs to torsion pairs using the truncation functor defined in the previous section. We will need the following generalization of Wakamatsu Lemma to extriangulated categories.

\begin{lemma}\label{WakLem}\cite[Lemma~3.1]{LZ}
    Let $\mathcal A$ be an extension closed subcategory of an extriangulated category $\C$ and $ t \xrightarrow{f} a\to b$ be a conflation such that $f$ is the minimal left $\mathcal A$-approximation of $t\in \C$. Then $b \in \prescript{\perp_1}{}{\mathcal A}$.
\end{lemma}

The following theorems \ref{tors-cotors} and \ref{bijections} generalise the existing bijection between torsion classes in $\mmod\La$ and cotorsion classes in $\K^{[-1,0]}(\proj\La)$ (\cite[Proposition~3.2]{PZ}, \cite{G2}). 

\begin{theorem}\label{tors-cotors}
    Let $\Y$ be a cotorsion class in $\K^{[-d+1,0]}(\proj\La)$. Then $\tau_{\geq -d+2}\Y$ is a torsion class in $\D^{[-d+2,0]}(\mmod\La)$. Moreover, if $\Y$ is hereditary (resp. complete, resp. hereditary and contravariantly finite), then $\tau_{\geq -d+2}\Y$ is positive (resp. covariantly finite, resp. an $s$-torsion class).
\end{theorem}
\begin{proof}
    Set $\Y':=\tau_{\geq -d+2}\Y$ and $\X'=(\tau_{\geq -d+2}\Y)^{\perp}$. We need to show that $\prescript{\perp}{}{\X'}=\Y'$. By definition, $\Y'\subseteq\prescript{\perp}{}{\X'}$. Let $Y'\in \prescript{\perp}{}{\X'}$. Then there exists $Y\in\K^{[-d+1,0]}(\proj\La)$ such that $\tau_{\geq -d+2}(Y)\cong Y'$. We want to show that $Y\in\Y$.

    Note that, since $(\prescript{\perp_1}{}{\Y},\Y)$ is a cotorsion pair, 
    
    \begin{equation*}
    \begin{split}
        \prescript{\perp_1}{}{\Y}&=\{X\in \K^{[-d+1,0]}(\proj\La)\mid \Eb(X,\Y)=0\}\\
            &=\{X\in \K^{[-d+1,0]}(\proj\La)\mid \Hom(\tau_{\geq-d+2}\Y, \Sigma^{-1}\tau_{\leq -1}\nu X)=0\}\ \ \ \ \ \mbox{(by Lemma \ref{Extlem})}\\
            &=\{X\in \K^{[-d+1,0]}(\proj\La)\mid \Sigma^{-1}\tau_{\leq -1}\nu X\in\X'\},
    \end{split}
    \end{equation*}
    and 
    \begin{equation*}
    \begin{split}
        \Y&=\{Z\in \K^{[-d+1,0]}(\proj\La)\mid \Eb(\prescript{\perp_1}{}{\Y},Z)=0\}\\
            &=\{Z\in \K^{[-d+1,0]}(\proj\La)\mid \Hom(\tau_{\geq-d+2}Z, \Sigma^{-1}\tau_{\leq -1}\nu \prescript{\perp_1}{}{\Y})=0\}\ \ \ \ \mbox{(by Lemma \ref{Extlem})}.
    \end{split}
    \end{equation*}
    Since, for all $X\in\prescript{\perp_1}{}{\Y}$, $\Sigma^{-1}\tau_{\leq -1}\nu X\in\X'$, and $Y'\in \prescript{\perp}{}{\X'}$, we get that $Y\in \Y$. Thus, $(\Y',\X')$ is a torsion pair in $\D^{[-d+2,0]}(\mmod\La)$.

    Now, suppose that $\Y$ is hereditary, which is equivalent to $\Y$ being stable under cones (Lemma \ref{s-implies-h}). Let $g:Z\to Z'\in \Y'$. Then there exists $f:Y\to Y'\in\Y$ such that $\tau_{\geq -d+2}f\cong g$. Using Lemma \ref{inflation}, we can assume that $C(f)\in \K^{[-d+1]}(\proj\La)$. Since $\Y$ is closed under cones, $C(f)\in \Y$. Using Lemma \ref{techlem}, $\tau_{\geq -d+2}C(g)\cong \tau_{\geq-d+2}C(f)\in \Y'$. Using Lemma \ref{p-implies-s}, we conclude that $\Y'$ is a positive torsion class. 

    Next, if $\Y$ is complete, then it is covariantly finite in $\K^{[-d+1,0]}(\proj\La)$. This implies that $\Y'$ is covariantly finite in $\D^{[-d+2,0]}(\mmod\La)$ (Lemma \ref{approximations}).

    Finally, suppose $\Y$ is hereditary and contravariantly finite. Then $\Y'$ is positive and contravariantly finite in $\D^{[-d+2,0]}(\mmod\La)$. Let $Z\in \D^{[-d+2,0]}(\mmod\La)$ and $f:Y'\to Z$ a minimal right $\Y'$-approximation of $Z$. Consider the triangles $Y'\xrightarrow{f}Z\xrightarrow{g}C(f)$ and $\tau_{\leq-d+1}C(f)\to C(f)\xrightarrow{\pi}\tau_{\geq -d+2}C(f)$ in $\D^b(\mmod\La)$. Then $\Ho^0(g)$ is an epimorphism and $\Ho^0(\pi)$ is an isomorphism, which implies that $\Ho^{0}(\pi \circ g)$ is an epimorphism. Thus, the cocone of $\pi \circ g$, say $Y$, is in $\D^{[-d+2,0]}(\mmod\La)$, and we get a conflation $Y\xrightarrow{f'} Z\xrightarrow{\pi\circ g} \tau_{\geq -d+2}C(f)$ in $\D^{[-d+2,0]}(\mmod\La)$. We will show that $Y\in\Y'$. 
        
    Using the octahedral axiom as shown below, 
    \begin{center}
    \begin{tikzcd}
    Z \arrow[r, "g"] \arrow[d, Rightarrow, no head] & C(f) \arrow[r] \arrow[d, "\pi"]                            & \Sigma Y' \arrow[d, "\Sigma h"]  \\Z \arrow[r, "\pi\circ g"] & \tau_{\geq -d+2}C(f) \arrow[d] \arrow[r]                   & \Sigma Y \arrow[d]\\& \Sigma \tau_{\leq -d+1}C(f) \arrow[r, Rightarrow, no head] & \Sigma \tau_{\leq -d+1}C(f)
    \end{tikzcd}
    \end{center}
    we get a triangle $Y'\xrightarrow{h}Y\to \tau_{\leq -d+1}C(f)$. Let $F\in \X'$. Then we have an exact sequence $\Hom(\tau_{\leq -d+1}C(f), F)\to \Hom(Y,F)\to \Hom(Y',F)=0$. Since $F\in\D^{[-d+2,0]}(\mmod\La)$, $\Hom(\tau_{\leq -d+1}C(f), F)=0$. Thus, $\Hom(Y,F)=0$ and $Y\in \prescript{\perp}{}{\X'}=\Y'$.
        
    Since $Y\in \Y'$ and $f$ is a minimal right $\Y'$-approximation of $Z$, there exists a map $h':Y\to Y'$ such that $fh'=f'$, which implies that $fh'h=f'h=f$. Since $f$ is minimal, we get that $f':Y\to Z\cong [0,f]:Y''\oplus Y'\to Z$ for some $Y''\in \D^b(\mmod\La)$. This implies that $C(f')\cong \tau_{\geq -d+2}C(f)\cong C(f)\oplus \Sigma Y''$. This gives that $C(f)\in \D^{[-d+2,0]}(\mmod\La)$. Using Lemma \ref{WakLem} for the conflation $Y'\xrightarrow{f}Z\xrightarrow{g}C(f)$ in the extriangulated category $\D^{[-d+2,0]}(\mmod\La)$, we get that $C(f)\in \X'$. Hence, $(\Y',\X')$ is an $s$-torsion pair in $\D^{[-d+2,0]}(\mmod\La)$.
\end{proof}

The above theorem gives us a poset homomorphism $\Phi: \cotors \K^{[-d+1,0]}(\proj\La)\to \tors\D^{[-d+2,0]}(\mmod\La)$.

\subsection{The bijections}
We now show that the above map $\Phi$ is a poset isomorphism.

\begin{theorem}\label{bijections}
The map $\Phi: \cotors \K^{[-d+1,0]}(\proj\La)\to \tors\D^{[-d+2,0]}(\mmod\La)$ defined as $\Y\mapsto \tau_{\geq -d+2}\Y$ is a poset isomorphism. Moreover, it restricts to an isomorphism between the following subposets.
\begin{enumerate}
    \item The set of hereditary cotorsion classes and the set of positive torsion classes. 
    \item The set of complete cotorsion classes and the set of covariantly finite torsion classes. 
    \item The set of hereditary contravariantly finite cotorsion classes and the set of $s$-torsion classes. 
\end{enumerate}
\end{theorem}
\begin{proof}
    We claim that the poset homomorphism 
    \begin{equation*}
        \begin{split}
            \Psi: \tors\D^{[-d+2,0]}(\mmod\La)&\to \cotors \K^{[-d+1,0]}(\proj\La)\\
            \T&\mapsto \{Y\in \K^{[-d+1,0]}(\proj\La)\mid \tau_{\geq-d+2}Y\in\T\}
        \end{split}
    \end{equation*}is the inverse of $\Phi$. Let $\T\in \tors\D^{[-d+2,0]}(\mmod\La)$. Set $\Y=\{Y\in \K^{[-d+1,0]}(\proj\La)\mid \tau_{\geq-d+2}Y\in\T\}$, and $\X:=\prescript{\perp_1}{}{\Y}$.
    \medskip
    \textbf{Step 1}. The map $\Psi$ is well-defined: For this, we need to prove that $\X^{\perp_1}=\Y$. By definition, $\Y\subseteq\X^{\perp_1}$. Using Lemma \ref{Extlem}, we know that $\X=\{X\in \K^{[-d+1,0]}(\proj\La)\mid \Sigma^{-1}\tau_{\leq -1}\nu X\in\T^\perp\}$, and $\X^{\perp_1}=\{Z\in \K^{[-d+1,0]}(\proj\La)\mid \Hom(\tau_{\geq-d+2}Z, \Sigma^{-1}\tau_{\leq -1}\nu \X)=0\}$. Let $Y\in \X^{\perp_1}$. In order to show that $Y\in \Y$, we need to show that $\tau_{\geq -d+2}Y\in\T=\prescript{\perp}{}{(\T^\perp)}$. Let $F\in\T^\perp$. Then $\Sigma F\in \D^{[-d+1,-1]}(\mmod\La)$. Using Corollary \ref{dual}, there exists $I\in \K^{[-d+2,1]}(\inj\La)$ such that $\tau_{\leq0}I\cong F$. Thus, $\tau_{\leq -1}\Sigma I\cong \Sigma F$. This implies that $\Sigma^{-1}\tau_{\leq -1}\nu(\nu^{-1}\Sigma I)=F\in\T^\perp$. Thus, $\nu^{-1}\Sigma I\in\X$. Since $\Y\in\X^{\perp_1}$, $\Hom(\tau_{\geq -d+2}Y, \Sigma^{-1}\tau_{\leq -1}\nu(\nu^{-1}\Sigma I))=\Hom(\tau_{\geq -d+2}Y, F)=0$. This implies that $\tau_{\geq -d+2}Y\in\T$, and we are done. 
    \medskip
    \textbf{Step 2}. The maps $\Phi$ and $\Psi$ are mutual inverses: This follows from the fact that $\tau_{\geq -d+2}$ is full and essentially surjective with the kernel $\add\La[d-1]$ contained in any cotorsion class, and that $\K^{[-d+1,0]}(\proj\La)$ is a $\Hom$-finite Krull-Schmidt category.
    \medskip
    \textbf{Step 3}. The map $\Psi$ sends positive torsion classes to hereditary cotorsion classes: Suppose $\T$ is positive. Then, using Lemma \ref{p-implies-s}, for all morphisms $g\in\T$, $\tau_{\geq -d+2}C(g)\in\T$. Let $y\xrightarrow{f}y'\to y''$ be a conflation in $\K^{[-d+1,0]}(\proj\La)$ with $y,y'\in\Y$. By definition, $\tau_{\geq -d+2}y$, $\tau_{\geq -d+2}y'\in \T$. Moreover, $\tau_{\geq -d+2}(C(\tau_{\geq -d+2}f))\in\T$. Thus, it follows from Lemma \ref{techlem} that $\tau_{\geq -d+2}y''\cong \tau_{\geq -d+2}(C(\tau_{\geq -d+2}f))\in\T$ and $y''\in\Y$. This shows that $\Y$ is closed under cones, and hence, by Lemma \ref{s-implies-h}, $\Y$ is a hereditary cotorsion class. 
    \medskip
    \textbf{Step 4}. The map $\Psi$ sends covariantly finite torsion classes to complete cotorsion classes: Suppose that $\T$ is covariantly finite. Then using Lemma \ref{approximations}, $\Y$ is also covariantly finite. We now show that $\Y$ is extension closed. Let $y\xrightarrow{f}y'\to y''$ be a conflation in $\K^{[-d+1,0]}(\proj\La)$ with $y,y''\in\Y$. By definition, this implies that $\tau_{\geq -d+2}y$, $\tau_{\geq -d+2}y''\in \T$. Moreover, using Lemma \ref{techlem}, we know that $\tau_{\geq -d+2}(C(\tau_{\geq -d+2}f))\cong \tau_{\geq -d+2}y''$. Using the triangle $$\tau_{\leq -d+1}(C(\tau_{\geq -d+2}f))\to C(\tau_{\geq -d+2}f)\to \tau_{\geq -d+2}(C(\tau_{\geq -d+2}f)),$$ we get that for all $F\in\F$, $\Hom(C(\tau_{\geq -d+2}f),F)=0$. Thus, using the triangle $$\tau_{\geq -d+2}y\xrightarrow{\tau_{\geq -d+2}f}\tau_{\geq -d+2}y'\to C(\tau_{\geq -d+2}f),$$ we get that $\Hom(\tau_{\geq -d+2}y',F)=0$ for all $F\in\F$. Thus, $\tau_{\geq -d+2}y'\in \prescript{\perp}{}{F}=\T$ and $y'\in \Y$.
    
    Let $C\in\K^{[-d+2,0]}(\proj\La)$. Let $f:C\to Y$ be the minimal $\Y$-approximation of $C$. This gives a conflation $C\xrightarrow{f}Y\to C(f)$ in $\K^{[-d+1,0]}(\proj\La)$. Since $\Y$ is closed under extensions using Lemma \ref{WakLem}, we get that $C(f)\in \prescript{\perp_1}{}{\Y}=\X$. Thus, for every $C\in \K^{[-d+2,0]}(\proj\La)$, there exists a conflation $C\to Y\to C(f)$ with $Y\in\Y$ and $C(f)\in\X$.

    Finally, let $C\in \K^{[-d+1,0]}(\proj\La)$. We have triangles $$\Sigma^{-1}t_{\leq -d+1}C\to t_{\geq -d+2}C\to C$$ and $$\Sigma^{-1}t_{\leq -1}C\to t_{\geq 0}C\to C$$ with $\Sigma^{-1}t_{\leq -1}C, t_{\geq -d+2}C\in \K^{[-d+2,0]}(\proj\La)$. Using the above argument, we know that there exist conflations $t_{\geq -d+2}C\to y\to x$ and $\Sigma^{-1}t_{\leq -1}C\to y'\to x'$ in $\K^{[-d+1,0]}(\proj\La)$ with $y,y'\in\Y$ and $x,x'\in \X$. Using the octahedral axiom for the following diagrams, 
    \begin{center}
        \begin{tikzcd}
        & \Sigma^{-1}x \arrow[d] \arrow[r, Rightarrow, no head] & \Sigma^{-1}x \arrow[d] &   \\
        \Sigma^{-1}t_{\leq -d+1}C \arrow[r] \arrow[d, Rightarrow, no head] & t_{\geq -d+2}C \arrow[d] \arrow[r]                    & C \arrow[d] \arrow[r]  & t_{\leq -d+1}C \arrow[d, Rightarrow, no head] \\
        \Sigma^{-1}t_{\leq -d+1}C \arrow[r]                                & y \arrow[d] \arrow[r]                                 & d \arrow[d] \arrow[r]  & t_{\leq -d+1}C \\& x \arrow[r, Rightarrow, no head] & x   & \end{tikzcd} 
    \end{center}
        \begin{center}
           \begin{tikzcd}
            & \Sigma^{-1}x' \arrow[d] \arrow[r, Rightarrow, no head] & \Sigma^{-1}x' \arrow[d]        &  \\
            \Sigma^{-1}C \arrow[r] \arrow[d, Rightarrow, no head] & \Sigma^{-1}t_{\leq-1}C \arrow[d] \arrow[r]             & t_{\geq0}C \arrow[d] \arrow[r] & C \arrow[d, Rightarrow, no head] \\\Sigma^{-1}C \arrow[r]   & y' \arrow[d] \arrow[r]                                 & e \arrow[d] \arrow[r]          & C  \\& x' \arrow[r, Rightarrow, no head]& x'          &            \end{tikzcd} 
        \end{center}
    we get conflations $C\to d\to x$ and $y'\to e\to C$. Since $y, t_{\leq -d+1}C\in\Y$, and $\Y$ is closed under extensions, we get that $d\in \Y$. Similarly, since $x', t_{\geq 0}C\in\X$, and $\X$ is closed under extensions, we get that $e\in \X$. Thus, using Remark \ref{complete}$, (\X,\Y)$ is a complete cotorsion pair.
    \medskip
    \textbf{Step 5}. The map $\Psi$ sends $s$-torsion classes to hereditary contravariantly finite cotorsion classes: Suppose $\T$ is an $s$-torsion class. Since an $s$-torsion class is positive and contravariantly finite, by Step 3 and Lemma \ref{approximations}, $\Y$ is hereditary and contravariantly finite in $\K^{[-d+1,0]}(\proj\La)$.
\end{proof}
\section{Silting objects, cotorsion pairs, and torsion pairs}
The main goal of this section is to prove the following theorem.
\begin{theorem}\label{main}
The map $\Phi$ defined in Theorem \ref{tors-cotors} and the maps $\psi$ and $\psi'$ given in Definition \ref{psi} and Definition \ref{psi2} restrict to give the following commutative triangle of poset isomorphisms.  
\begin{equation}\label{res-diag}
    \begin{tikzcd}
    d\mbox{-}\silt\La \arrow[rr, "\psi"] \arrow[rrd, "\psi'"'] &  & \chcotors \K^{[-d+1,0]}(\proj\La)\arrow[d, "\Phi"] \\ &  & \fptors \D^{[-d+2,0]}(\mmod\La)                     
    \end{tikzcd}
    \end{equation}
\end{theorem}

\subsection{Definition of \texorpdfstring{$\psi$}{psi}}
\begin{definition}\label{psi}
    Let $M$ be a $d$-term silting object in $\K^b(\proj\La)$. Define
    \begin{align*}
    \X_M:=&\mbox{ the smallest full subcategory of }\K^b(\proj \La)\mbox{ containing } \{\Sigma^m M\mid m\leq 0\}\\&\mbox{ and closed under extensions } \mbox{ and summands,}\\
    \Y_M:=&\mbox{ the smallest full subcategory of }\K^b(\proj \La)\mbox{ containing } \{\Sigma^m M\mid m\geq 0\}\\&\mbox{ and closed under extensions } \mbox{ and summands.}
    \end{align*}
    Set $\psi(M):=(\X_M\cap \K^{[-d+1,0]}(\proj\La),\Y_M\cap \K^{[-d+1,0]}(\proj\La))$.
\end{definition}
Using \cite[Theorem~6.1]{KY}, we know that the map $M\mapsto (\X_M,\Y_M)=:\Theta(M)$ gives a poset isomorphism from $\silt\K^b(\proj \La)$ to the poset of bounded co-$t$-structures on $\K^b(\proj \La)$. Moreover, the inverse of this map is given by taking the additive generator of the co-heart of a bounded co-$t$-structure. Set $\X'_M:=\X_M\cap \K^{[-d+1,0]}(\proj\La)$ and $\Y'_M=\Y_M\cap \K^{[-d+1,0]}(\proj\La)$.

\begin{lemma}\label{partiallemma}
    The map $\psi$ given in Definition \ref{psi} induces an injective poset homomorphism $\psi: d\mbox{-}\silt\La\to \chcotors\K^{[-d+1,0]}(\proj\La)$. 
\end{lemma}
\begin{proof}
    We first show that $\psi$ is well-defined. Since $(\X_M,\Y_M)$ is a co-$t$-structure on $\K^b(\proj \La)$, $\Eb(\X'_M,\Y'_M)=0$. 

    Moreover, since $M$ is a $d$-term silting object, and $\K^{\leq 0}(\proj\La)$ and $\K^{\geq -d+1}(\proj\La)$ are full subcategories of $\K^b(\proj \La)$ closed under extensions and summands containing $\{\Sigma^m M\mid m\geq 0\}$ and $\{\Sigma^m M\mid m\leq 0\}$ respectively, we get that $\Y_M\subseteq \K^{\leq 0}(\proj\La)$ and $\X_M\subseteq \K^{\geq -d+1}(\proj\La)$.



    Let $C\in \K^{[-d+1,0]}(\proj\La)$. Then there exists a triangle $T\xrightarrow{f} \Sigma C\xrightarrow{g} F$ with $T\in \X_M$ and $F\in \Sigma \Y_M$. Thus, $F\in \K^{\leq -1}(\proj\La)$ and $T\in \K^{\geq -d+1}(\proj\La)$. Since $\K^{\leq 0}(\proj\La)$ is extension closed, using the triangle $\Sigma^{-1}F\to T\to \Sigma C$, we get that $T\in \K^{\leq 0}(\proj\La)$. Similarly, using the triangle $\Sigma C \to F \to \Sigma T$, we get that $F\in \K^{\geq -d}(\proj\La)$. Thus, we get a triangle $C\to \Sigma^{-1}F\to T$ with $\Sigma^{-1}F\in\Y'_M$ and $T\in \X'_M$.

    Let $C\in \K^{[-d+1,0]}(\proj\La)$. Then there exists a triangle $T\xrightarrow{f} C\xrightarrow{g} F$ with $T\in \X_M$ and $F\in \Sigma \Y_M$. Thus, $F\in \K^{\leq -1}(\proj\La)$ and $T\in \K^{\geq -d+1}(\proj\La)$. Since $\K^{\geq -d}(\proj\La)$ is extension closed, using the triangle $X\to F\to\Sigma T$, we get that $F\in \K^{\geq -d}(\proj\La)$. Thus, $\Sigma^{-1}F\in \K^{[-d+1,0]}(\proj\La)$. Similarly, using the triangle $\Sigma^{-1}F\to T\to C$, we get that $T\in \K^{[-d+1,0]}(\proj\La)$. Thus, we get a triangle $\Sigma^{-1}F\to T\to C$ with $\Sigma^{-1}F\in\Y'_M$ and $T\in \X'_M$.

    Thus, using Remark \ref{complete}$, (\X'_M,\Y'_M)$ is a complete cotorsion pair in $\K^{[-d+1,0]}(\proj\La)$. Since $\X_M$ and $\Y_M$ are closed under cocones and cones respectively, so are $\X'_M$ and $\Y'_M$, which implies that $(\X'_M,\Y'_M)$ is hereditary by Lemma \ref{s-implies-h}. 

    Now, suppose $M\leq N$ in $d$-$\silt\La$. Then $\X_M\supseteq \X_N$ which implies that $\X'_M\supseteq\X'_N$. Thus, $\psi$ is a poset homomorphism.

    We finally show that $\psi$ is injective. Let $M,M'\in d\mbox{-}\silt\La$ such that $\Y'_M=\Y'_{M'}$. Using the bijection between silting objects and co-$t$-structures \cite[Theorem~6.1]{KY}, it is enough to show that $\Y_M=\Y_{M'}$. Suppose not. Without loss of generality, let $Y\in \Y_M\setminus\Y_{M'}$. Since $\Y_M\subseteq \K^{\leq 0}(\proj\La)$, $Y\in \K^{\leq 0}(\proj\La)$. We have a triangle $\Sigma^{-1}t_{\leq -d}Y\to t_{\geq -d+1}Y\to Y$. Since $\X_M, \X_{M'}\subseteq \K^{\geq -d+1}(\proj\La)$, we get that $\Sigma^{-1}t_{\leq -d}Y\in \X_M^{\perp_1}\cap \X_{M'}^{\perp_1}=\Y_M\cap \Y_{M'}$. Thus, $t_{\geq -d+1}Y\in \Y_M$, since $\Y_M$ is closed under extensions. Since $t_{\geq -d+1}Y\in \K^{[-d+1,0]}(\proj\La)$, we get that $t_{\geq -d+1}Y\in \Y'_M=\Y'_{M'}$. Since $\Y_{M'}$ is closed under cones, we get that $Y\in \Y_{M'}$, a contradiction. Hence $\Y_M=\Y_{M'}$ which implies that $M\cong M'$.
    
\end{proof}
\begin{theorem}
        The map $\psi$ is an isomorphism of posets.
    \end{theorem}
    \begin{proof}
    We will show this by constructing an inverse poset homomorphism $\chi:\chcotors\K^{[-d+1,0]}(\proj\La)\to d\mbox{-}\silt\La $ of $\psi$.

    Let $\Y'$ be a complete hereditary cotorsion class in $\K^{[-d+1,0]}(\proj\La)$. Set $\X'=\prescript{\perp_1}{}{\Y'}$. Define $C(\X')$ to be the smallest full subcategory of $\K^b(\proj \La)$ closed under summands and extensions containing $\X'$ and $\K^{\geq 0}(\proj\La)$ and $C(\Y')$ to be the smallest full subcategory of $\K^b(\proj \La)$ closed under summands and extensions containing $\Y'$ and $\K^{\leq -d+1}(\proj\La)$. Note that $C(\X')\subseteq\K^{\geq -d+1}(\proj\La)$ and $C(\Y')\subseteq \K^{\leq 0}(\proj\La)$. We will show that $(C(\X'),C(\Y'))$ is a bounded co-$t$-structure on $\K^b(\proj\La)$ whose co-heart lies in $\K^{[-d+1,0]}(\proj\La)$, and set $\chi(\Y'):=\Theta^{-1}(C(\X'),C(\Y'))$.
    
    Using the fact that both $\Eb(\X',\Y')$ and $\Eb(\X',\K^{\leq -d+1}(\proj\La))$ vanish, and that $\X'^{\perp_1}$ in $\K^b(\proj\La)$ is closed under extensions and summands, we get that $$\Eb(\X',C(\Y'))=0.$$ Similarly, $\Eb(\K^{\geq 0}(\proj\La),C(\Y'))=0$. Thus, $\X',\K^{\geq 0}(\proj\La)\subseteq \prescript{\perp_1}{}{C(\Y')}$. Since $\prescript{\perp_1}{}{C(\Y')}$ is also closed under extensions and summands, we get that $$\Eb(C(\X'),C(\Y'))=0.$$

    We next show that $\Sigma^{-1}C(\X')\subseteq C(\X')$. Note that it is enough to show that $\Sigma^{-1}\X'\subseteq C(\X')$ and $\Sigma^{-1}\K^{\geq 0}(\proj\La)\subseteq C(\X')$. Clearly, $$\Sigma^{-1}\K^{\geq 0}(\proj\La)=\K^{\geq 1}(\proj\La)\subseteq\K^{\geq 0}(\proj\La)\subseteq C(\X').$$ Let $X\in \X'$. We have the following triangle in $\K^b(\proj\La)$. $$t_{\geq 1}\Sigma^{-1}X\to \Sigma^{-1}X\to t_{\leq 0}\Sigma^{-1}X\to t_{\geq 0}X\to X$$
    Since $X\in \K^{[-d+1,0]}(\proj\La)$, $t_{\geq 0}X\in \prescript{\perp_1}{}{\Y'}=\X'$. Since $\X'$ is closed under cocones, we get that $t_{\leq 0}\Sigma^{-1}X\in\X'$. Since $\Sigma^{-1}X$ is an extension of $t_{\geq 1}\Sigma^{-1}X\in \K^{\geq 0}(\proj\La)$ and $t_{\leq 0}\Sigma^{-1}X\in\X'$, we get that $\Sigma^{-1}X\in C(\X')$. Thus $\Sigma^{-1}C(\X')\subseteq C(\X')$. Dually, we can show that $\Sigma C(\Y')\subseteq C(\Y')$.

    Finally, let $C\in \K^b(\proj\La)$. Set $C':= t_{\geq -d+1}C$ and $C'':=t_{\leq 0}C'$. Since $t_{\leq 0}t_{\geq -d+1}C\in \K^{[-d+1,0]}(\proj\La)$, using the completeness of the cotorsion pair $(\X',\Y')$, we get that there exists a triangle $Y\to X\xrightarrow{f} t_{\leq 0}t_{\geq -d+1}C$ with $Y\in\Y'$ and $X\in\X'$. 

    Now, using the octahedral axiom for the triangles $X\xrightarrow{f} t_{\leq 0}t_{\geq -d+1}C\to\Sigma Y$, $t_{\leq 0}t_{\geq -d+1}C=t_{\geq -d+1}t_{\leq 0}C\xrightarrow{j} t_{\leq 0}C\to t_{\leq -d}t_{\leq 0}C=t_{\leq -d}C$, and $X\xrightarrow{j\circ f} t_{\leq 0}C\to C(j\circ f)$, we get a triangle $\Sigma Y\to C(j\circ f)\to t_{\leq -d}C$. Since $\Sigma Y, t_{\leq -d}C\in \Sigma C(\Y')$, we get that $C(j\circ f)\in\Sigma C(\Y')$. 

    \begin{center}
    \begin{tikzcd}
    X \arrow[r, "f"] \arrow[d, Rightarrow, no head] & t_{\leq 0}t_{\geq -d+1}C \arrow[r] \arrow[d, "j"] & \Sigma Y \arrow[d]    \\ X \arrow[r, "j\circ f"']& t_{\leq 0}C \arrow[d] \arrow[r, "l'"'] & C(j\circ f) \arrow[d] \\& t_{\leq -d}C \arrow[r, Rightarrow, no head]       & t_{\leq -d}C         
    \end{tikzcd}
    \end{center}

    Finally, using the octahedral axiom for the triangles $C\xrightarrow{l} t_{\leq 0}C\to \Sigma t_{\geq 1}C$, $t_{\leq 0}C\xrightarrow{l'} C(j\circ f)\to\Sigma X$, and $C\xrightarrow{l'\circ l}C(j\circ f)\to C(l'\circ l)$, we get a triangle $\Sigma t_{\geq 1}C'\to C(l'\circ l)\to \Sigma X$. 

    \begin{center}
    \begin{tikzcd}
    C \arrow[r, "l"] \arrow[d, no head, Rightarrow] & t_{\leq 0}C \arrow[r] \arrow[d, "l'"]   & \Sigma t_{\geq 1}C \arrow[d] \\
    C \arrow[r, "l'\circ l"']                       & C(j\circ f) \arrow[d] \arrow[r]         & C(l'\circ l) \arrow[d]       \\
    & \Sigma X \arrow[r, no head, Rightarrow] & \Sigma X        \end{tikzcd}
    \end{center}
    
    Since $\Sigma t_{\geq 1}C, \Sigma X\in\Sigma C(\X')$, we get that $C(l'\circ l)\in\Sigma C(\X')$. Thus, we have a triangle $\Sigma^{-1}C(l'\circ l)\to C\to C(j\circ f)$ with $\Sigma^{-1}C(l'\circ l)\in C(\X')$ and $C(j\circ f)\in\Sigma C(\Y')$. Therefore, $(C(\X'),C(\Y'))$ is a co-$t$-structure. Moreover, since $\K^{\geq 0}(\proj\La)\subseteq C(\X')$, we get that $\bigcup_{i\in\mathbb{Z}}\Sigma^iC(\X')\supseteq \bigcup_{i\in\mathbb{Z}}\Sigma^i\K^{\geq 0}(\proj\La)=\K^b(\proj\La)$. Similarly, $\bigcup_{i\in\mathbb{Z}}\Sigma^iC(\Y')=\K^b(\proj\La)$. Hence, $(C(\X'),C(\Y'))$ is a bounded co-$t$-structure on $\K^b(\proj\La)$. Moreover, the co-heart $C(\X')\cap C(\Y')$ is contained in $\K^{[-d+1,0]}(\proj\La)$, and hence its additive generator is a $d$-term silting object.

    Note that if $\Y\subseteq\Y'\in \chcotors\K^{[-d+1,0]}(\proj\La)$, then $C(\Y)\subseteq C(\Y')$, which implies that $\Theta^{-1}(C(\prescript{\perp_1}{}{\Y}), C(\Y))\leq \Theta^{-1}(C(\prescript{\perp_1}{}{\Y'}), C(\Y'))$. Thus, $\chi$ is a poset homomorphism.

    Since $\X'\subseteq C(\X')\cap \K^{[-d+1,0]}(\proj\La)$ and $\Y'\subseteq C(\Y')\cap \K^{[-d+1,0]}(\proj\La)$, we get that $\Y'=(\X')^{\perp_1}\supseteq C(\X')^{\perp_1}\cap \K^{[-d+1,0]}(\proj\La)=C(\Y')\cap \K^{[-d+1,0]}(\proj\La)$. The last equality follows from the fact that for a co-$t$-structure $(\mathcal A,\mathcal B)$, $\mathcal{A}^{\perp_1}=\mathcal{B}$ and $\mathcal{A}=\prescript{\perp_1}{}{\mathcal{B}}$ \cite[Proposition~1.6]{P}. Thus, $\Y'= C(\Y')\cap \K^{[-d+1,0]}(\proj\La)$. Similarly, $\X'= C(\X')\cap \K^{[-d+1,0]}(\proj\La)$. Thus, $\psi\circ \chi=\mbox{Id}$. This implies that $\psi\circ \chi\circ \psi=\psi$. Since $\psi$ is injective by Lemma \ref{partiallemma}, we get that $\chi\circ \psi=\mbox{Id}$. Thus, $\chi$ is the inverse of $\psi$.
    \end{proof}

\subsection{Definition of \texorpdfstring{$\psi'$}{psi'}}
\begin{definition}\label{psi2}
    Let $M$ be a $d$-term silting object in $\K^b(\proj\La)$. Define 
    \begin{align*}
    \U_M&=\{N\in \D^b(\mmod\La)\mid \Hom(M, \Sigma^m N)=0,\ \forall \ m>0\},\\
    \V_M&=\{N\in \D^b(\mmod\La)\mid \Hom(M, \Sigma^m N)=0,\ \forall \ m<0\}.
\end{align*}
Set $\psi'(M):=(\U_M\cap \D^{[-d+2,0]}(\mmod\La),\Sigma^{-1}\V_M\cap \D^{[-d+2,0]}(\mmod\La))$.
\end{definition}
Using \cite[Theorem~6.1]{KY} we know that the poset of equivalence classes of silting objects in $\K^b(\proj \La)$ is isomorphic to the poset of bounded $t$-structures with length heart on $\D^b(\mmod\La)$ under the map $M\mapsto (\U_M, \Sigma^{-1}\V_M)$. Set $\U'_M:=\U_M\cap \D^{[-d+2,0]}(\mmod\La)$ and $\V'_M=\Sigma^{-1}\V_M\cap \D^{[-d+2,0]}(\mmod\La)$.
\begin{lemma}\label{stors}
 The map $\psi':d\mbox{-}\silt\La\to \ptors\D^{[-d+2,0]}(\mmod\La)$ given by $M\mapsto \U'_M$ is a poset homomorphism, with the image contained in the set of $s$-torsion classes. In particular, the image is contained in the set of contravariantly finite positive torsion classes. 
\end{lemma}
\begin{proof}
We start by showing that for $M\in d\mbox{-}\silt\La$, $(\U'_M,\V'_M)$ is an $s$-torsion pair in $\D^{[-d+2,0]}(\mmod\La)$.

Since $(\U_M, \Sigma^{-1}\V_M)$ is a $t$-structure in $\D^b(\mmod\La)$, $\Hom(\U'_M, \V'_M)=0$. Moreover, since $\U_M$ is closed under positive shifts, $\Eb^{-1}(\U'_M, \V'_M)=0$.

In order to prove condition (Sc), we claim that $\D^{\leq (-d+1)}(\mmod\La)\subseteq\U_M$. This is because if $N\in \D^{\leq (-d+1)}(\mmod\La)$, then $N\cong \tau_{\leq -d+1}N$, which implies that
\begin{equation*}
    \begin{split}
        \Hom_{\D^b(\mmod\La)}(\Sigma^{-m}M, N)&\cong \Hom_{\D^b(\mmod\La)}(\Sigma^{-m}M, \tau_{\leq -d+1}N)\\
        &\cong \Hom_{K(\mmod\La)}(\Sigma^{-m}M, \tau_{\leq -d+1}N)\\ &=0
    \end{split}
\end{equation*}
for all $m>0$, since $M$ is a complex of projectives concentrated in $[-d+1, 0]$. Dually, we have that $\D^{\geq 1}(\mmod\La)\subseteq \Sigma^{-1}\V_M$.

Finally, let $Z\in \D^{[-d+2,0]}(\mmod\La)$. Since $(\U_M, \Sigma^{-1}\V_M)$ is a $t$-structure in $\D^b(\mmod\La)$, there exists a triangle $$U\xrightarrow{u}Z\xrightarrow{v}V$$ with $U\in\U_M$ and $V\in\Sigma^{-1}\V_M$. Note that we have a triangle $$\tau_{\leq 0}U\to U\to \tau_{\geq 1}U.$$ Using the previous paragraph, we know that $\tau_{\geq 1}U\in \Sigma^{-1}\V_M$, which implies that the map from $U\to \tau_{\geq 1}U$ is~$0$. Thus, $\tau_{\leq 0}U\cong U\oplus \Sigma^{-1}\tau_{\geq 1}U$. This implies that $\tau_{\geq 1}U\cong 0$ and $\tau_{\leq 0}U\cong U$. Thus $U\in \D^{\leq 0}(\mmod\La)$. Using the triangle $Z\to V\to \Sigma U$ and the fact that $\D^{\leq 0}(\mmod\La)$ is closed under extensions, we get that $V\in \D^{\leq 0}(\mmod\La)$. Similarly, we have a triangle $$\tau_{\leq -d+1}V\to V\to \tau_{\geq -d+2}V.$$ Since $\tau_{\leq -d+1}V\in \U_M$, we get that the map from $\tau_{\leq -d+1}V\to V$ is~$0$, which implies that $\tau_{\geq -d+2}V\cong V\oplus \Sigma\tau_{\leq -d+1}V$. This gives that $\tau_{\leq -d+1}V=0$ and $\tau_{\geq -d+2}V\cong V$. Thus, $V\in \D^{[-d+2,0]}(\mmod\La)$. Using the triangle $\Sigma^{-1}V\to U\to Z$ and the fact that $\D^{\geq -d+2}(\mmod\La)$ is closed under extensions, we get that $U\in \D^{[-d+2,0]}(\mmod\La)$. Hence, $U\in \U'_M$ and $V\in \V'_M$. Thus, $(\U_M',\V_M')$ is an $s$-torsion pair. Using \cite{AET} and Corollary \ref{pconts} we know that $s$-torsion classes are positive and contravariantly finite. Thus, the image of $\psi'$ is contained in the set of contravariantly finite positive torsion classes.

Now, suppose $M\leq N$. Since $M\mapsto (\U_M, \Sigma^{-1}V_M)$ is a map of posets, $\U_M\subseteq \U_N$. And hence $\U'_M\subseteq \U'_N$. Thus, $\psi'$ is a poset homomorphism. 
\end{proof}
We now have the following triangle of maps.
\begin{equation}\label{silt-tors-cotors}
\begin{tikzcd}
d\mbox{-}\silt\La \arrow[rr, "\psi"] \arrow[rrd, "\psi'"'] &  & \chcotors \K^{[-d+1,0]}(\proj\La)\arrow[d, "\Phi"] \\ &  & \ptors \D^{[-d+2,0]}(\mmod\La)                     
\end{tikzcd}
\end{equation}
\begin{proposition}\label{commute}
    The above triangle commutes.
\end{proposition}
\begin{proof}
    We need to show that for all $M\in d\mbox{-}\silt\La$, $\U_M'=\tau_{\geq -d+2}\Y_M'$. 

    Let $N\in \U_M'$ and $\cdots\to N^{-2}\to N^{-1}\to N^0\to 0\in\K^b(\proj\La)$ quasi-isomorphic to $N$. Then $\tau_{\geq -d+2}(N^{-d+1}\to\cdots\to N^0)\cong N$. We claim that $N^{-d+1}\to\cdots\to N^0\in \Y_M$. 

    Consider the triangle $t_{\geq -d+1}N\to N\to t_{\leq -d}N$. Let $j\leq 0$. Then $\Sigma^jM\in \X_M$. We have the following exact sequence.
    $$\Hom(\Sigma^jM, t_{\leq -d}N)\to \Eb(\Sigma^jM,t_{\geq -d+1}N)\to \Eb(\Sigma^jM, N)$$
    Since $N\in \U_M'$, $\Eb(\Sigma^jM, N)\cong \Hom(M, \Sigma^{-j+1}N)=0$. Moreover, since $M$ is concentrated in degrees $-d+1,\cdots, 0$, $\Hom(\Sigma^jM, t_{\leq -d}N)=0$. Therefore, $$\Eb(\Sigma^jM,t_{\geq -d+1}N)=0.$$ This implies that $\Sigma^jM\in \prescript{\perp_1}{}{(t_{\geq -d+1}N)}$ for all $j\leq 0$. Since, $\prescript{\perp_1}{}{(t_{\geq -d+1}N)}$ is closed under extensions and summands, we get that $\X_M\subseteq \prescript{\perp_1}{}{(t_{\geq -d+1}N)}$. Thus, $t_{\geq -d+1}N\in \X_M^{\perp_1}=\Y_M$. Thus, $\U_M'\subseteq \tau_{\geq -d+2}\Y_M'$.

    For the other inclusion, we first claim that $\Y_M'=t_{\geq -d+1}\Y_M$. It is enough to show that $\Y_M$ is closed under $t_{\geq -d+1}$. Let $Y\in\Y_M$. Consider the triangle $t_{\geq -d+1}Y\to Y\to t_{\leq -d}Y$. Since $\X_M$ is contained in $\K^{\geq -d+1}(\proj\La)$, we get that for all $X\in\X_M$, $\Eb(X, t_{\geq -d+1}Y)=0$, which implies that $t_{\geq -d+1}Y\in \X_M^{\perp_1}=\Y_M$.

    Thus, $\tau_{\geq -d+2}\Y_M'=\tau_{\geq -d+2}t_{\geq -d+1}\Y_M=\tau_{\geq -d+2}\Y_M$.

    Let $m>0$ and $j\geq 0$. Using the triangle $\tau_{\leq -d+1}\Sigma^j M\to \Sigma^j M\to \tau_{\geq -d+1}\Sigma^j M$, we get the exact sequence $$\Hom(\Sigma^{-m}M,\Sigma^jM)\to \Hom(\Sigma^{-m}M, \tau_{\geq -d+1}\Sigma^j M)\to \Eb(\Sigma^{-m}M, \tau_{\leq -d+1}\Sigma^j M).$$ Since $M$ is a silting object, $\Hom(\Sigma^{-m}M,\Sigma^jM)=0$. Moreover, since $M$ is concentrated in degrees $-d+1, \cdots, 0$ as a complex of projectives, $\Eb(\Sigma^{-m}M, \tau_{\leq -d+1}\Sigma^j M)=0$. Thus, $\Hom(M, \Sigma^{m}\tau_{\geq -d+1}\Sigma^j M)=0$ for all $m>0$, and $\tau_{\geq -d+1}\Sigma^j M\in\U_M'$ for all $j\geq 0$. 

    Finally, to show that $\tau_{\geq -d+2}\Y_M\subseteq\U_M'$, it is enough to show that the category $$S:=\{X\in \K^b(\proj\La)\mid \Hom(M, \Sigma^m\tau_{\geq -d+2}X)=0 \ \forall\ m>0\}$$ is closed under extensions and summands.

    Clearly, it is closed under summands. Let $X\xrightarrow{f} Z\to Y$ be a triangle in $\K^b(\proj\La)$ with $X,Y\in S$. Set $C=C(\tau_{\geq -d+2}f)$. Then $\tau_{\geq -d+2}C\cong \tau_{\geq -d+2}Y$. Let $m>0$. Using the triangle $\tau_{\leq -d+1}C\to C\to \tau_{\geq -d+2}C$, we get the exact sequence $$\Hom(M, \Sigma^m\tau_{\leq -d+1}C)\to \Hom(M, \Sigma^mC)\to \Hom(M, \Sigma^m\tau_{\geq -d+2}C),$$ such that $\Hom(M, \Sigma^m\tau_{\leq -d+1}C)=\Hom(M, \Sigma^m\tau_{\geq -d+2}C)=0$. Thus, $$\Hom(\Sigma^{-m}M, C)\cong\Hom(M, \Sigma^mC)=0.$$ Using the triangle $\tau_{\geq -d+2}X\to \tau_{\geq -d+2}Z\to C$ and the fact that $(\Sigma^{-m}M)^\perp$ is closed under extensions, we get that $\Hom(M, \Sigma^m\tau_{\geq -d+2}Z)=0$ for all $m>0$. Thus, $Z\in S$.
    \end{proof}

    \subsection{Proof of Theorem \ref{main}}
    \begin{proof}
        Using Lemma \ref{stors}, we know that the image of $\psi'$ is contained in the class of contravariantly finite torsion pairs. Using Proposition \ref{commute}, this image is also contained in the image of $\Phi$, which by Theorem \ref{tors-cotors} is contained in the class of covariantly finite torsion pairs. Thus, we get that the following commutative triangle is well-defined. 
        \begin{equation*}
    \begin{tikzcd}
    d\mbox{-}\silt\La \arrow[rr, "\psi"] \arrow[rrd, "\psi'"'] &  & \chcotors \K^{[-d+1,0]}(\proj\La)\arrow[d, "\Phi"] \\ &  & \fptors \D^{[-d+2,0]}(\mmod\La)                     
    \end{tikzcd}
    \end{equation*}
    Using Theorem \ref{bijections}, we know that the inverse $\Psi:\tors \D^{[-d+2,0]}(\mmod\La)\to \cotors\K^{[-d+1,0]}(\proj\La)$ of $\Phi$ sends functorially finite torsion pairs to complete hereditary cotorsion pairs. Hence, $$\Phi:\chcotors \K^{[-d+1,0]}(\proj\La)\to \fptors \D^{[-d+2,0]}(\mmod\La)$$ is a poset isomorphism. 

    Thus, we conclude that all the maps in triangle (\ref{res-diag}) are isomorphisms of posets.
    \end{proof}
    
    \begin{corollary}
        Let $(\T,\F)$ be a positive torsion pair in $\D^{[-d+2,0]}(\mmod\La)$. Then $\T$ is functorially finite if and only if $\F$ is.
    \end{corollary}
    \begin{proof}
    Using Theorem \ref{main} and its dual, we know that $\T$ is functorially finite if and only if it $\T=\psi'(P)$ for $P$ a $d$-term silting complex, which is true if and only if $\F=(\psi'(P))^\perp$ is functorially finite.  
    \end{proof}

    \begin{proposition}
       The class of functorially finite positive torsion classes in \linebreak $\D^{[-d+2,0]}(\mmod\La)$ coincides with the class of functorially finite $s$-torsion classes in $\D^{[-d+2,0]}(\mmod\La)$.
    \end{proposition}
    \begin{proof}
        Suppose $\T$ is a functorially finite positive torsion class. Then, using Theorem \ref{bijections}, $\Phi^{-1}(\T)$ is a hereditary cotorsion class. Since $\T=\tau_{\geq -d+2}(\Phi^{-1}(\T))$ is contravariantly finite, Lemma \ref{approximations} implies that $\Phi^{-1}(\T)$ is contravariantly finite. Then Theorem \ref{bijections} gives that $\T$ is a functorially finite $s$-torsion class. On the other hand, the class of functorially finite $s$-torsion classes is contained in the class of functorially finite positive torsion classes by Corollary \ref{pconts}.
    \end{proof}

    \subsection{Lattices}
    Let $(\C,\Eb,\s)$ be an extriangulated category.
    \begin{theorem}\label{torslattice}
        The poset of torsion pairs in $\C$ is a lattice.
    \end{theorem}
    \begin{proof}
        Let $(\T_1,\F_1)$, $(\T_2,\F_2)$ be two torsion pairs. To show that their meet exists, it is enough to show that $(\T_1\cap \T_2, (\T_1\cap \T_2)^\perp)$ is a torsion pair, i.e.,  $\T_1\cap\T_2=\prescript{\perp}{}{((\T_1\cap\T_2)^\perp)}$. By definition, $\T_1\cap\T_2\subseteq\prescript{\perp}{}{((\T_1\cap\T_2)^\perp)}$. On the other hand, for $i=1,2$, since $\T_1\cap \T_2\subseteq \T_i$, we get that $\prescript{\perp}{}{((\T_1\cap\T_2)^\perp)}\subseteq \prescript{\perp}{}{(\T_i^\perp)}=\T_i$. Thus, $\prescript{\perp}{}{((\T_1\cap\T_2)^\perp)}\subseteq \T_1\cap \T_2$. Hence $\T_1\cap\T_2=\prescript{\perp}{}{((\T_1\cap\T_2)^\perp)}$.

        Dually, the join of $(\T_1,\F_1)$, $(\T_2,\F_2)$ is the torsion pair $(\prescript{\perp}{}{(\F_1\cap\F_2)},\F_1\cap\F_2)$.
    \end{proof}
    A similar proof can be applied to obtain the corresponding result for cotorsion pairs in $\C$.
    \begin{theorem}\label{cotorslattice}
        The poset of cotorsion pairs in $\C$ is a lattice.
    \end{theorem}
    \begin{proof}
        Let $(\X_1,\Y_1)$, $(\X_2,\Y_2)$ be two cotorsion pairs in $\C$. To show that their join exists, it is enough to show that $(\X_1\cap \X_2, (\X_1\cap \X_2)^{\perp_1})$ is a cotorsion pair, i.e.,  $\X_1\cap\X_2=\prescript{\perp_1}{}{((\X_1\cap\X_2)^{\perp_1})}$. By definition, $\X_1\cap\X_2\subseteq\prescript{\perp_1}{}{((\X_1\cap\X_2)^{\perp_1})}$. On the other hand, for $i=1,2$, since $\X_1\cap \X_2\subseteq \X_i$, we get that $\prescript{\perp_1}{}{((\X_1\cap\X_2)^{\perp_1})}\subseteq \prescript{\perp_1}{}{(\X_i^{\perp_1})}=\X_i$. Thus, $\prescript{\perp_1}{}{((\X_1\cap\X_2)^{\perp_1})}\subseteq \X_1\cap \X_2$. Hence $\X_1\cap\X_2=\prescript{\perp_1}{}{((\X_1\cap\X_2)^{\perp_1})}$.

        Dually, the meet of $(\X_1,\Y_1)$, $(\X_2,\Y_2)$ is the cotorsion pair $(\prescript{\perp_1}{}{(\Y_1\cap\Y_2)},\Y_1\cap\Y_2)$.
    \end{proof}
    Moreover, in the case of $\C=\D^{[-d+2,0]}(\mmod\La)$, we get that the subposet of positive torsion pairs in $\C$ is also a lattice.
    \begin{theorem}\label{ptorslattice}
        The subposet $\ptors \D^{[-d+2,0]}(\mmod\La)$ of $\tors\D^{[-d+2,0]}(\mmod\La)$ is closed under meets and joins, and hence, a lattice. 
    \end{theorem}
    \begin{proof}
        Using the characterization of positive torsion pairs in Lemma \ref{p-implies-s}, we get that positive torsion classes and positive torsion-free classes are stable under intersections.
    \end{proof}
    \begin{corollary}\label{hcotorslattice}
        The subposet $\hcotors\K^{[-d+1,0]}(\proj\La)$ is a lattice.
    \end{corollary}
    \section{Examples}
    \begin{example}\label{nsemid}
        Let $Q=1\to 2$ and $\La=KQ$. Then the Hasse diagram of the poset $3$-$\silt\La$ is shown below.
        
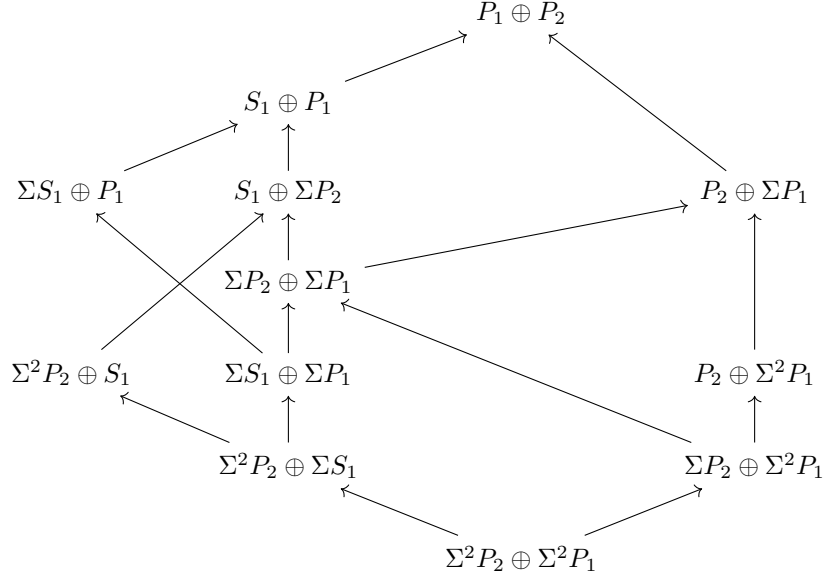
\begin{center}
        \begin{tikzcd}
        &  & P_1\oplus P_2  &  \\ & S_1\oplus P_1 \arrow[ru]   & &  \\ \Sigma S_1\oplus P_1 \arrow[ru]  & S_1 \oplus \Sigma P_2\arrow[u]& & P_2\oplus \Sigma P_1 \arrow[luu] \\ & \Sigma P_2 \oplus \Sigma P_1\arrow[rru] \arrow[u] &  &  \\ \Sigma^2 P_2\oplus S_1 \arrow[ruu] & \Sigma S_1\oplus \Sigma P_1\arrow[u] \arrow[luu] &  & P_2\oplus \Sigma^2 P_1 \arrow[uu]   \\   & \Sigma^2 P_2\oplus\Sigma S_1 \arrow[u] \arrow[lu]  &  & \Sigma P_2\oplus \Sigma^2 P_1 \arrow[u] \arrow[lluu] \\  &   & \Sigma^2 P_2\oplus \Sigma^2 P_1 \arrow[lu] \arrow[ru] & 
        \end{tikzcd}
        \end{center}
        Since $\D^{[-1,0]}(\mmod\La)$ has finitely many indecomposables, this is also the lattice of positive torsion classes in $\D^{[-1,0]}(\mmod\La)$. This example also shows that, unlike the case of (positive) torsion classes in a module category (\cite[Theorem~3.1]{DIRRT}), the lattice of positive torsion classes in $\D^{[-d,0]}(\mmod\La)$ is not necessarily semidistributive. 
    \end{example}
    \begin{example}
        Let $Q$ be the quiver \begin{tikzcd}
        1 \arrow[r, "\alpha", bend left] & 2 \arrow[l, "\beta", bend left]
        \end{tikzcd} and $I=\langle \alpha\beta,\beta\alpha\rangle$. Then the Hasse diagram of the poset $3\mbox{-}\silt kQ/I$ is as follows.
        
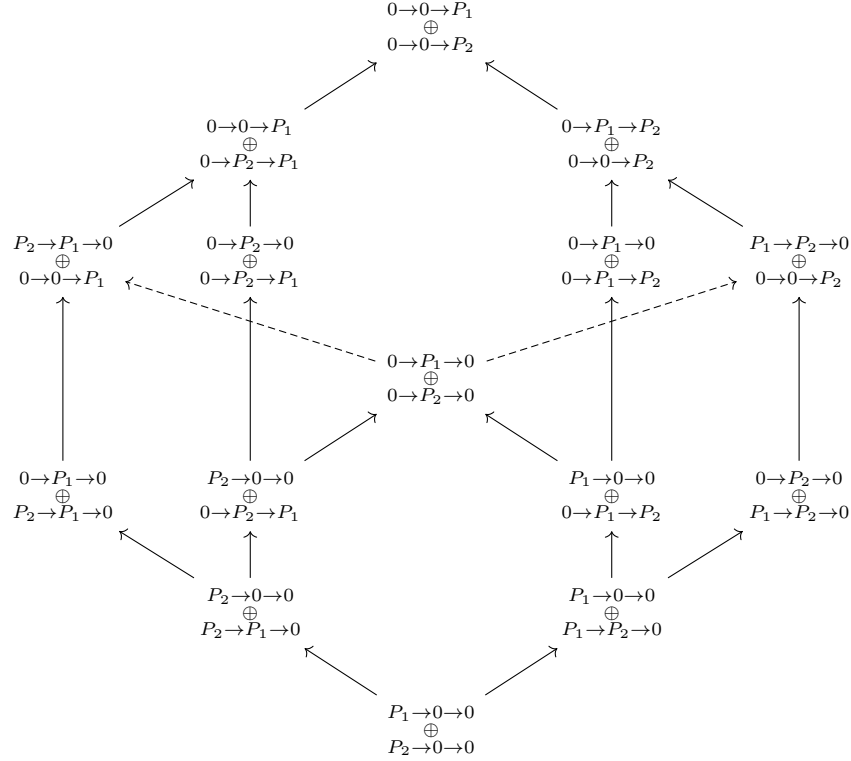
\begin{center}
        \begin{tikzcd}
        & &\substack{0\to 0\to P_1\\ \oplus\\ 0\to 0\to P_2}&  & \\
        &\substack{0\to 0\to P_1\\ \oplus\\ 0\to P_2\to P_1}\arrow[ru] & & \substack{0\to P_1\to P_2\\ \oplus\\ 0\to 0\to P_2}\arrow[lu] & \\
        \substack{P_2\to P_1\to 0\\ \oplus\\ 0\to 0\to P_1} \arrow[ru] & \substack{0\to P_2\to 0\\ \oplus\\ 0\to P_2\to P_1} \arrow[u] & & \substack{0\to P_1\to 0\\ \oplus\\ 0\to P_1\to P_2} \arrow[u]  & \substack{P_1\to P_2\to 0\\ \oplus\\ 0\to 0\to P_2} \arrow[lu] \\
        &  & \substack{0\to P_1\to 0\\ \oplus\\ 0\to P_2\to 0}\arrow[llu, dashed] \arrow[rru, dashed] && \\
        \substack{0\to P_1\to 0\\ \oplus\\ P_2\to P_1\to 0} \arrow[uu] & \substack{P_2\to 0\to 0\\ \oplus\\ 0\to P_2\to P_1}\arrow[ru] \arrow[uu] & & \substack{P_1\to 0\to 0\\ \oplus\\ 0\to P_1\to P_2} \arrow[lu] \arrow[uu] & \substack{0\to P_2\to 0\\ \oplus\\ P_1\to P_2\to 0}\arrow[uu]\\
        & \substack{P_2\to 0\to 0\\ \oplus\\ P_2\to P_1\to 0} \arrow[u] \arrow[lu]  && \substack{P_1\to 0\to 0\\ \oplus \\ P_1\to P_2\to 0} \arrow[u] \arrow[ru]  & \\
        & & \substack{P_1\to 0\to 0\\ \oplus \\ P_2\to 0\to 0 }\arrow[lu] \arrow[ru]&  &                          \end{tikzcd}
        \end{center}
    \end{example}
    We also show that the class of $s$-torsion pairs can be strictly contained in the class of positive torsion pairs.
    \begin{example}
        Let $Q=\begin{tikzcd}
1 \arrow[r, shift left] \arrow[r, shift right] & 2
\end{tikzcd}$ be the Kronecker quiver. Let $\mathcal{P}$ and $\mathcal{I}$ denote the preprojective and the postinjective component of the AR-quiver of $KQ$. The positive torsion classes in $\C=\D^{[-1,0]}(\mmod\La)$ are given by the additive subcategories generated by the following sets. 
\begin{enumerate}
    \item Any final part of $\Sigma\mathcal{I}$.
    \item Shift of a subset of the tubes and $\Sigma\mathcal I$.
    \item Any final part of $\mathcal I\cup\Sigma\mathcal P$, the shift of all tubes, and $\Sigma\mathcal I$.
    \item $\Sigma S_2$.
    \item $M$ and the final part of $\Sigma\mathcal{I}$ starting at $\Sigma M$, $M\in\mathcal I$.
    \item Any subset $S$ of the tubes, the shift of any subset $S'$ containing $S$, and $\mathcal{I}$.
    \item Any subset of the tubes, shifts of all tubes, and $\mathcal I\cup\Sigma\mathcal P\cup\Sigma\mathcal{I}$.
    \item Any final subset of $\mathcal{P}$, all tubes and their shifts, $\Sigma\mathcal{P}\cup\mathcal{I}\cup\Sigma\mathcal{I}$.
    \item $S_2,\Sigma S_2$.
    \item $M$ and the final part of $\Sigma\mmod\La$ starting at $\Sigma M$, $M\in\mathcal{P}$. 
\end{enumerate}
The Hasse diagram of the lattice $\ptors\C$ is depicted in the following picture.
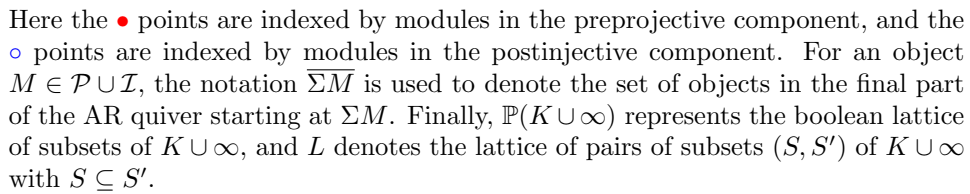
\begin{center}
    \begin{tikzcd}
    &  & &   & \C   &               &                                               \\   &     &                              &                                                                 & & {\langle \mathrm{ind}\ \C\setminus[01]\rangle} \arrow[lu, no head]           &   \\ & &                              &  &                                                                        & \vdots \arrow[u, no head]                                         &  \\ & &                              & && \fbox{$\mathbb{P}(K\cup\infty)$} \arrow[u, no head]                       &   \\{\langle[01],\overline{\Sigma[01]}\rangle} \arrow[rrrruuuu, no head]   & {\color{red}\bullet\ \bullet\ \bullet} \arrow[rrrruu, no head] \arrow[rrrruuu, no head] &                              & && \vdots \arrow[u, no head]&\\&  & {\color{blue}\circ\ \circ\ \circ} \arrow[rrru, no head] & {\langle[10],\overline{\Sigma[10]}\rangle} \arrow[rru, no head] & & \langle\Sigma\mmod\La\rangle \arrow[u, no head]                   & \fbox{$\mathbb{P}(K\cup\infty)$}\arrow[luu, no head] \\ {\langle[01],\Sigma[01]\rangle} \arrow[uu, no head] &&  & & & \vdots \arrow[u, no head]  & \\{\langle\Sigma[01]\rangle} \arrow[rrrrruu, no head] \arrow[u, no head] &            & &   &  & \fbox{$\mathbb{P}(K\cup\infty)$} \arrow[u, no head] \arrow[ruu, no head] &\\
    &&& & & \vdots \arrow[u, no head]&         \\ &  &   &  &    & {\langle\Sigma[21],\Sigma[10]\rangle} \arrow[u, no head]          &\\&                                &                              & && {\langle\Sigma[10]\rangle} \arrow[u, no head]                     &\\ &                                &                              &  & 0 \arrow[lllluuuu, no head] \arrow[ru, no head] \arrow[luuuuuu, no head] \arrow[lluuuuuu, no head] \arrow[llluuuuuuu, no head] \arrow[lllluuuuuuu, no head] & &                                             
\end{tikzcd}
\end{center}
Here the $\color{red}\bullet$ points are indexed by modules in the preprojective component, and the $\color{blue}\circ$ points are indexed by modules in the postinjective component. For an object $M\in \mathcal{P}\cup\mathcal{I}$, the notation $\overline{\Sigma M}$ is used to denote the set of objects in the final part of the AR quiver starting at $\Sigma M$. Finally, $\mathbb{P}(K\cup\infty)$ represents the boolean lattice of subsets of $K\cup\infty$, and $L$ denotes the lattice of pairs of subsets $(S,S')$ of $K\cup\infty$ with $S\subseteq S'$.

Note that the positive torsion class $\T=\mathrm{add}\{\mathcal{S}\cup \Sigma \mathcal{S}\cup \Sigma\mathcal{I}\}$, where $\mathcal{S}$ is a proper subset of the tubes, is not an $s$-torsion class as the objects in $\mathcal{I}$ do not admit a right $\T$-approximation.  
\end{example}

\section{\texorpdfstring{$d$}{d}-torsion classes in \texorpdfstring{$d$}{d}-cluster tilting subcategories}\label{cluster-tilting}
In this section, we will recall the definitions of a $d$-cluster tilting subcategory $\mathcal{M}$ in $\mmod\La$ and $d$-torsion classes in $\M$. These have been studied in \cite{AHJKPT2} where the authors show the existence of an injective map from functorially finite $(d-1)$-torsion classes to $d$-term silting complexes. 

\begin{definition}
    A functorially finite subcategory $\M\subseteq\mmod\La$ is called $d$-cluster tilting if 
    \begin{enumerate}
        \item $\M=\{X\in\mmod\La\mid \Ext^i(X,\M)=0\ \forall\  i=1,\cdots,d-1\}$
        \item $\M=\{X\in\mmod\La\mid \Ext^i(\M,X)=0\ \forall\  i=1,\cdots,d-1\}$
    \end{enumerate}
\end{definition}
\begin{definition}
    Let $\M\subseteq\mmod\La$ be a $d$-cluster tilting subcategory. A subcategory $\U\subseteq\M$ is called a \emph{$d$-torsion class} if for every $M\in\M$, there exists an exact sequence $$0\to U_\M\to M\to V_1\to \cdots\to V_d\to 0$$ with $V_i\in\M$ for all $i=1,\cdots,d$ such that 
    \begin{enumerate}
        \item $U_M\in \U$;
        \item For all $U\in U$, the sequence $$0\to \Hom(U,V_1)\to \cdots\to \Hom(U,V_d)\to 0$$ is exact. 
    \end{enumerate}
\end{definition}
\begin{theorem}\cite[Theorem~8.1, Proposition~8.16]{AHJKPT2}
There exists an injective map 
\begin{equation*}
    \left\{ \parbox{11em}{\centering functorially finite $(d-1)$-torsion classes in $\M$} \right\} \xhookrightarrow{} \left\{\parbox{11em}{\centering $d$-term silting complexes in $\K^b(\proj\La)$}\right\}
\end{equation*}
that sends $\U\mapsto P^{M_\U}\oplus P_{\U}[d-1]$, where $M_\U$ is the basic $\Ext^{d-1}$-projective generator of $\U$, $P^{M_\U}$ is its minimal projective $(d-1)$-presentation, and $P_\U$ is the maximal basic projective $\La$-module such that $\Hom(P_\U,\U)=0$. 

Moreover, the silting complex $P^\bullet_\U:=P^{M_\U}\oplus P_{\U}[d-1]$ is \emph{inner acyclic}, i.e., $\Ho^i(P^\bullet_\U)=0$ for all $i\neq 0,-d+1$.
\end{theorem}

Using the dual of Theorem \ref{main}, we know that $d$-term silting objects are in bijection with functorially finite positive torsion-free classes in $\D^{[-d+2,0]}(\mmod\La)$. We want to study the composition of the above inclusion with this bijection. We need the following lemma. 

\begin{lemma}
Let $(\T,\F)$ be a torsion pair in $\D^{[-d+2,0]}(\mmod\La)$. Then $\F\,\cap\,\mmod\La$ is a torsion-free class in $\mmod\La$, and $\T\,\cap\,(\mmod\La[d-2])$ is a torsion class in $\mmod\La[d-2]$.
\end{lemma}
\begin{proof}
We only prove the torsion-free case. The torsion case is dual. 

We claim that $\F\cap\mmod\La=\{X\in \mmod\La\mid \Hom(\tau_{\geq 0}\T,X)=0\}$. Indeed, for $X\in \F\cap\mmod\La$ and $Y\in \T$, applying $\Hom(-,X)$ to the triangle $$Y\to \tau^{\geq 0}Y\to \Sigma\tau_{\leq -1}Y,$$ we get the exact sequence $0=\Hom(\Sigma\tau_{\leq -1}Y, X)\to \Hom(\tau^{\geq 0}Y,X)\to \Hom(X,Y)=0$. Thus, $\Hom(\tau^{\geq 0}Y,X)=0$. On the other hand, if there exists $X\in \mmod\La$ such that $\Hom(\tau_{\geq 0}\T,X)=0$, then $\Hom(\T,X)=0$. This is because for $Y\in \T$, we have the triangle $\tau_{\leq -1}Y\to Y\to \tau^{\geq 0}Y$. Since $\Hom(\tau_{\leq -1}Y,X)=0=\Hom(\tau^{\geq 0}Y,X)$, we get that $\Hom(Y,X)=0$. Thus, $X\in \T^\perp=\F$.   

Thus, we showed that $\F\cap\mmod\La$ is the right orthogonal of some subcategory of $\mmod\La$. Hence, it is a torsion-free class. 
\end{proof}

\begin{proposition}
   The following diagram commutes.
   \begin{center}
       \begin{tikzcd}[column sep=large]
       \left\{ \parbox{12em}{\centering functorially finite $(d-1)$-torsion classes in $\M$} \right\}\arrow[r, hook, "\U\mapsto P^\bullet_\U"]\arrow[d, "\U\mapsto \U^\perp"]& \left\{\parbox{11em}{\centering $d$-term silting complexes in $\K^b(\proj\La)$}\right\}\arrow[d, "\phi"', "\sim" sloped]\\ 
       \left\{ \parbox{11em}{\centering torsion-free classes in $\mmod\La$} \right\} &\left\{\parbox{11em}{\centering functorially finite positive torsion-free classes in $\D^{[-d+2,0]}(\mmod\La)$}\right\}\arrow[l,"\F\cap\mmod\La\,\mapsfrom\,\F"]
       \end{tikzcd}
   \end{center}
\end{proposition}
\begin{proof}
Let $\U$ be a functorially finite $(d-1)$-torsion class in $\M$. Then $P^\bullet_\U=P^{M_\U}\oplus P_{\U}[d-1]$. By definition, $\phi(P^\bullet_\U)=\{X\in \D^{[-d+2,0]}(\mmod\La)\mid \Hom(P^\bullet_\U,\Sigma^iX)=0 \ \forall \ i\leq 0\}$. We want to show that $\phi(P^\bullet_\U)\cap \mmod\La=\U^\perp$. 

Let $X\in \U^\perp$. Since $X\in \mmod\La$, $\Hom(P_\U[d-1],X)=0=\Hom(P^{M_\U},\Sigma^iX)$ for all $i<0$. Note that $P^{M_\U}=t_{\geq -d+1}M^\bullet$, where $M^\bullet$ is a minimal projective resolution of $M_\U$. Using the triangle $t_{\geq -d+1}M^\bullet\to M^\bullet\to t_{\leq -d}M^\bullet$, since $\Hom(\Sigma^{-1}t_{\leq -d}M^\bullet, X)=0=\Hom(M^\bullet,X)$, we get that $\Hom(t_{\geq -d+1}M,X)=0$. Thus $\Hom(P^\bullet_\U,X)=0$, and $X\in\phi(P^\bullet_\U)\cap \mmod\La$. 

Now suppose that $X\in \phi(P^\bullet_\U)\cap \mmod\La$ and $A\in \U$. Since $A\in \mmod\La$, there exists an epimorphism $g:\Lambda^n\to A$ for some $n\geq 1$. Using \cite[Theorem~4.6]{AHJKPT2}, there exists a minimal $\U$-coresolution of $\Lambda$ $$0\to \Lambda\xrightarrow{f}U^0\to U^1\to\cdots\to U^{d-1}\to 0$$ with $U^i$ being $\Ext^{d-1}$-projective in $\U$ for $i=0,\cdots,d-1$. Since $f$ is a left $\U$-approximation of $\La$, there exists an epimorphism $h: (U^0)^n\to A$ such that the following triangle commutes. 
\begin{center}
\begin{tikzcd}
\Lambda^n \arrow[rd, "g"', two heads] \arrow[r, "f\oplus\cdots\oplus f"] & (U^0)^n \arrow[d, "h", two heads, dashed] \\  & A                        \end{tikzcd}    
\end{center}
Since $X\in \phi(P^\bullet_\U)\cap \,\mmod\La$, $\Hom(P^{M_\U}, X)=0$. Again using the triangle $t_{\geq -d+1}M^\bullet\to M^\bullet\to t_{\leq -d}M^\bullet$, we get that $\Hom(M^\bullet,X)=0$ since $\Hom(t_{\geq -d+1}M^\bullet,X)=0=\Hom(t_{\leq -d}M^\bullet,X)$. Since $M^\bullet\cong M_\U$ is the $\Ext^{d-1}$-projective generator of $\U$, $U^0\in \add M_\U$. Thus, $\Hom(U^0, X)=0$. Since $A$ is a quotient of $(U^0)^n$, we get that $\Hom(A,X)=0$. Hence, $X\in \U^\perp$. 
\end{proof}
Note that the torsion class $\T$ corresponding to $\U^\perp$ is the smallest torsion class in $\mmod\La$ containing $\U$, and hence satisfies $\T\cap \M=\U$ (\cite[Theorem~1.1]{AJST}). 
\printbibliography

@book{H1, 
    series={London Mathematical Society Lecture Note Series},
    title={Triangulated Categories in the Representation of Finite Dimensional Algebras}, 
    publisher={Cambridge University Press}, 
    author={Happel, Dieter},
    year={1988}, 
}

@article{DIJ,
    author = {Demonet, Laurent and Iyama, Osamu and Jasso, Gustavo},
    title = "{$\tau$-Tilting Finite Algebras, Bricks, and $g$-Vectors}",
    journal = {International Mathematics Research Notices},
    volume = {2019},
    number = {3},
    pages = {852-892},
    year = {2017},    
    doi = {10.1093/imrn/rnx135},
}

@article{AI,
    author = {Aihara, Takuma and Iyama, Osamu},
    title = "{Silting mutation in triangulated categories}",
    journal = {Journal of the London Mathematical Society},
    volume = {85},
    number = {3},
    pages = {633-668},
    year = {2012}

}

@article{T,
    author = {Thomas, Hugh},
    title = "{An introduction to the lattice of torsion classes}",
    journal = {Bull. Iran. Math. Soc.},
    volume = {47},
    pages = {35-55},
    year = {2021}

}

@article{KY,
    author    = "Steffen Koenig and Dong Yang",
    title     = "Silting objects, simple-minded collections, $t$-structures and co-$t$-structures for finite-dimensional algebras",
    year      = "2014",
    %addendum = "(accessed: 31.03.2020)",
    journal   = "Doc. Math.",
    volume={19},
    pages ={403-438}
}

@article{AIR,
    title={$\tau $-tilting theory}, 
    volume={150},
    number={3},
    journal={Compositio Mathematica},
    publisher={London Mathematical Society},
    author={Adachi, Takahide and Iyama, Osamu and Reiten, Idun},
    year={2014},
    pages={415–452}
}

@article{AET,
    title={Intervals of s-torsion pairs in extriangulated categories with negative first extensions}, 
    volume={174}, 
    DOI={10.1017/S0305004122000354}, 
    number={3}, 
    journal={Mathematical Proceedings of the Cambridge Philosophical Society}, 
    publisher={Cambridge University Press}, 
    author={Adachi, Takahide and Enomoto, Haruhisa and Tsukamoto, Mayu}, 
    year={2023}, 
    pages={451–469}
}

@article{NP,
  TITLE = {{Extriangulated categories, Hovey twin cotorsion pairs and model structures}},
  AUTHOR = {Nakaoka, Hiroyuki and Palu, Yann},
  JOURNAL = {{Cahiers de topologie et g{\'e}om{\'e}trie diff{\'e}rentielle cat{\'e}goriques}},
  VOLUME = {LX},
  NUMBER = {2},
  PAGES = {117--193},
  YEAR = {2019}
}

@misc{AHJKPT2,
      title={Higher torsion classes, $\tau_d$-tilting theory and silting complexes}, 
      author={Jenny August and Johanne Haugland and Karin M. Jacobsen and Sondre Kvamme and Yann Palu and Hipolito Treffinger},
      year={2026},
      eprint={2602.03659},
      archivePrefix={arXiv},
      primaryClass={math.RT},
}

@misc{AHJKPT1,
    author = {August, Jenny and Haugland, Johanne and Jacobsen, Karin and Kvamme, Sondre and Palu, Yann and Treffinger, Hipolito},
    title =  "A characterisation of higher torsion classes",
    year={2023},
      eprint={2301.10463},
      archivePrefix={arXiv},
      primaryClass={math.RT}
}

@article{STW,
author = {Stump, Christian and Thomas, Hugh and Williams, Nathan},
title = {Cataland: Why the Fuss?},
volume = {DMTCS Proceedings, FPSAC 2016},
journal = {Discrete Mathematics \& Theoretical Computer Science},
year={2020}}

@article{PZ,
author = {Pauksztello, David and Zvonareva, Alexandra},
title = {Co-$t$-structures, cotilting and cotorsion pairs},
volume = {175},
number={1},
journal = {Mathematical Proceedings of the Cambridge Philosophical Society},
year={2023},
pages={89–106}}

@article{IJ,
author = {Iyama, Osamu and Jin, Haibo},
title = {Positive Fuss–Catalan Numbers and Simple-Minded Systems in Negative Calabi–Yau Categories},
volume = {2023},
number={8},
journal = { International Mathematics Research Notices},
year={2023},
pages={6624–6647}}

@misc{LZ,
    title={Hereditary cotorsion pairs on extriangulated subcategories}, 
      author={Liu, Yu and Zhou, Panyue },
      year={2020},
      eprint={2012.06997},
      archivePrefix={arXiv},
      primaryClass={math.RT}
}

@misc{G1,
    title={On thick subcategories of the category of projective presentations}, 
      author={Garcia, Monica },
      year={2023},
      eprint={2303.05226},
      archivePrefix={arXiv},
      primaryClass={math.RT}
}

@misc{G2,
      title={On $g$-finiteness in the category of projective presentations}, 
      author={Garcia, Monica},
      year={2024},
      eprint={2406.04134},
      archivePrefix={arXiv},
      primaryClass={math.RT}
}

@misc{GNP,
      title={Positive and negative extensions in extriangulated categories}, 
      author={Gorsky, Mikhail and Nakaoka, Hiroyuki and Palu, Yann},
      year={2021},
      eprint={2103.12482},
      archivePrefix={arXiv},
      primaryClass={math.CT}
}

@article{AT,
author = {Adachi, Takahide and Tsukamoto, Mayu},
title = {Hereditary cotorsion pairs and silting subcategories in extriangulated categories},
volume = {594},
journal = {Journal of Algebra},
year={2022},
pages={109-137}}

@article{DIRRT,
author = {Demonet, Laurent and Iyama, Osamu and Reading, Nathan and Reiten, Idun and Thomas, Hugh},
title = {Lattice theory of torsion classes: Beyond $\tau$-tilting theory},
volume = {10},
journal = {Transactions of the American Mathematical Society, Series B},
year={2023},
pages={542-612}}

@article{MS,
    author = {Marks, Frederik and Šťovíček, Jan},
    title = {Torsion classes, wide subcategories and localisations},
    journal = {Bulletin of the London Mathematical Society},
    year = {2017},
    volume = {49},
    number={3},
pages={405-416}
}

@article{IT, 
title={Noncrossing partitions and representations of quivers}, 
volume={145}, 
number={6}, 
journal={Compositio Mathematica}, 
author={Ingalls, Colin and Thomas, Hugh}, 
year={2009}, 
pages={1533–1562}}

@article{BBDG, 
title={Faisceaux pervers}, 
volume={100}, 
journal={Astérisque}, 
author={Beilinson, Alexander and Bernstein, Joseph and Deligne, Pierre and Gabber, Ofer}, 
year={1982}, 
pages={5–171}}

@article{P, 
title={Compact corigid objects in triangulated categories and co-$t$-structures}, 
volume={6}, 
number={1},
journal={Central European Journal of Mathematics}, 
author={Pauksztello, David}, 
year={2008}, 
pages={25-42}}

@article{B, 
title={Weight structures vs. t-structures; weight filtrations, spectral sequences, and complexes (for motives and in general)}, 
volume={6}, 
number={3},
journal={Journal of K-theory}, 
author={Bondarko, M.V.}, 
year={2010}, 
pages={387–504}}

@article{D,
 author = {Dickson, Spencer E.},
 journal = {Transactions of the American Mathematical Society},
 number = {1},
 pages = {223--235},
 title = {A Torsion Theory for Abelian Categories},
 volume = {121},
 year = {1966}
}

@article{XY,
 author = {Xu, Jinde and Yang, Yichao},
 journal = {Archiv der Mathematik},
 number = {4},
 pages = {383--389},
 title = {A Bongartz-type lemma for silting complexes over a hereditary algebra},
 volume = {114},
 year = {2020}
}

@article{BR,
author = {Beligiannis, Apostolos and Reiten, Idun},
year = {2007},
title = {Homological and homotopical aspects of torsion theories},
volume = {883},
journal = {Memoirs of the American Mathematical Society}
}

@article{N,
  title={General Heart Construction on a Triangulated Category (I): Unifying t-Structures and Cluster Tilting Subcategories},
  author={Nakaoka, Hiroyuki},
  journal={Applied Categorical Structures},
  year={2009},
  volume={19},
  pages={879-899}
}

@article{H2,
  title={Cotorsion pairs, model category structures, and representation theory},
  author={Hovey, Mark},
  journal={Mathematische Zeitschrift},
  year={2002},
  volume={241},
number={3},
  pages={553--592}
}

@article{St,
  title={Exact model categories, approximation theory, and
cohomology of quasi-coherent sheaves},
  author={Š\v{t}ovíček, Jan},
  journal={ Advances in representation
theory of algebras},
  year={ 2013},
  volume={EMS Ser. Congr. Rep. },
number={ Eur. Math. Soc.,
Zurich},
  pages={297--367}
}

@article{KQ,
  title={Exchange graphs and Ext quivers},
  author={King, Alastair  and Qiu, Yu},
  journal={Advances in Mathematics},
  year={2015},
  volume={285},
  pages={1106--1154}
}

@article{AJST, 
title={On higher torsion classes}, 
volume={248}, 
journal={Nagoya Mathematical Journal}, 
author={Asadollahi, Javad and Jørgensen, Peter and Schroll, Sibylle and Treffinger, Hipolito}, 
year={2022}, 
pages={823–848}
}

@article{Sa,
    author = {Luigi Salce},
    title = {Cotorsion theories for abelian groups},
    journal = {Symposia Mathematica},
    volume = {23},
    pages = {11--32},
    year = {1979}
}
\Address

\end{document}